\documentclass[a4paper]{amsart}
\usepackage{amsmath,amssymb,amsthm,enumerate,graphicx}%%graphicx is for rotatebox.
\title[Global CR invariants]
{Renormalized characteristic forms of the Cheng--Yau metric and global CR invariants}
      
\author{TAIJI MARUGAME}
\date{}

\newcommand\R{\mathbb{R}}
\newcommand\C{\mathbb{C}}
\newcommand\Ric{{\rm Ric}}
\renewcommand\a{\alpha}
\renewcommand\b{\beta}
\newcommand\g{\gamma}
\renewcommand\d{\delta}
\newcommand\e{\epsilon}
\renewcommand\r{\rho}
\renewcommand\th{\theta}
\newcommand\br{\boldsymbol{\rho}}
\newcommand\bh{\boldsymbol{h}}
\newcommand\bth{\boldsymbol{\theta}}
\newcommand\U{\Upsilon}
\newcommand\pa{\partial}
\newcommand\ol{\overline}
\newcommand{\calE}{\mathcal{E}}

\newcommand{\wt}{\widetilde}
\newcommand{\wh}{\widehat}

\newtheorem{lem}{Lemma}[section]
\newtheorem{thm}[lem]{Theorem}
\newtheorem{theorem}[lem]{Theorem}
\newtheorem{prop}[lem]{Proposition}
\newtheorem{cor}[lem]{Corollary}
\theoremstyle{definition}

\newtheorem{rem}[lem]{\it Remark}

\numberwithin{equation}{section}

\makeatletter
\@addtoreset{equation}{section}

\makeatother
\address{Mathematical Analysis Team, RIKEN Center for Advanced Intelligence Project (AIP), 1-4-1 Nihonbashi, Chuo-ku, Tokyo 103-0027, Japan}
\address{Department of Mathematics, Graduate School of Science, Osaka University, 1-1 Machikaneyama-cho Toyonaka Osaka 560-0043, Japan}
\email{taiji.marugame@riken.jp}



\begin{document}

\begin{abstract} 
For each invariant polynomial $\Phi$, we construct a global CR invariant via the renormalized characteristic form of the Cheng--Yau metric on a strictly pseudoconvex domain. When the degree of $\Phi$ is 0, the invariant agrees with the total $Q'$-curvature. When the degree is equal to the CR dimension, we construct a primed pseudo-hermitian invariant $\mathcal{I}'_\Phi$ which integrates to the corresponding CR invariant. These are generalizations of the $\mathcal{I}'$-curvature on CR five-manifolds, introduced by Case--Gover. 
\end{abstract}
\maketitle

%%%%%%%%%%%%%%%%%%%%%%%%%%%%%%%%%%%%%%%%%%%%%%%%%%%%%%%%%%%%%%%
\section{Introduction}
The Cheng--Yau metric $g$ is a complete K\"ahler--Einstein metric on a bounded strictly pseudoconvex domain $\Omega\subset\C^{n+1}$, given by the K\"ahler form $-i\pa\ol\pa\log\r$ with a defining function $\r$ which solves the complex Monge--Amp\`ere equation \cite{ChY}. Since $g$ is biholomorphically invariant, one may try to construct biholomorphic invariants of $\Omega$ or CR invariants of the boundary $M$ by using geometric quantities of this metric. However, due to the singularity of $g$ at the boundary, we need some renormalization procedure to extract finite values from invariants of $g$.

Burns--Epstein \cite{BE2} introduced such a renormalization for the Levi-Civita connection of $g$. Let $\psi_i{}^j$ be the connection 1-forms of $g$ in a $(1, 0)$-coframe 
$\{\th^1, \dots, \th^{n+1}\}$. They defined the renormalized connection $\ol\nabla$ by the connection 1-forms
\begin{equation}\label{theta-intro}
\theta_i{}^j:=\psi_i{}^j+\frac{1}{\rho}(\rho_k\d_i{}^j+\rho_i \d_k{}^j)\theta^k\quad (\pa\r=\r_i\th^i),
\end{equation}
and showed that it extends to a linear connection on $\ol\Omega$. The transformation \eqref{theta-intro} is so-called a c-projective transformation \cite{CEMN}, and this renormalization procedure is an example of {\it c-projective compactifications}, introduced by \v Cap--Gover \cite{CapG}.

Burns--Epstein also defined the renormalized curvature, which is continuous up to $M$, by 
setting
\[
W_i{}^j:=\Psi_i{}^j+(g_{k\ol l}\d_i{}^j+g_{i\ol l}\d_k{}^j)\th^k\wedge \th^{\ol l},
\]
where $\Psi_i{}^j$ is the curvature form of $g$. Since $g$ satisfies ${\rm Ric}(g)+(n+2)g=0$, this agrees with both the c-projective Weyl curvature and the Bochner curvature of the Cheng--Yau metric. Also, this coincides with the $(1, 1)$-component of the curvature $\Theta_i{}^j$ of $\ol\nabla$. For an Ad-invariant polynomial $\Phi$ on $\mathfrak{gl}(n+1, \mathbb{C})$, they showed $\Phi(W)=\Phi(\Theta)$ and called it the {\it renormalized characteristic form}. When $\Phi$ has degree $n+1$, the integral of $\Phi(W)$ over $\Omega$ converges to a biholomorphic invariant of $\Omega$, which is called the {\it characteristic number} of the domain. By using the Chern--Simons transgression and the Lee--Melrose expansion of the Monge--Amp\`ere solution, they proved that this number is determined by local geometric data of the boundary.  
However, since the transgression is based on the global coordinates of $\C^{n+1}$, they needed a complicated topological procedure to relate the invariant to intrinsic CR geometry of $M$, and it does not provide a global CR invariant which can be written in terms of Tanaka--Webster curvature quantities.

As an exceptional case, they derived a Chern--Gauss--Bonnet formula for $c_2(\Theta)$ whose boundary term gives a global invariant of CR three-manifolds, called the {\it Burns--Epstein invariant} \cite{BE1}. The author \cite{Mar} generalized the renormalized Chern--Gauss--Bonnet formula and the Burns--Epstein invariant to higher dimensional cases in a more general setting: We consider a strictly pseudoconvex domain $\Omega$ in a complex manifold $X$ of dimension $n+1$. If the boundary $M$ admits a pseudo-Einstein contact form $\th$, there exists a global approximate solution $\r$ to the Monge--Amp\`ere equation such that $\theta=(i/2)(\pa\r-\ol\pa\r)|_{TM}$, and we can define the Cheng--Yau metric on $\Omega$ as a hermitian metric $g$ which agrees with $-i\pa\ol\pa\log\r$ near $M$; see \S\ref{ambient-CY} for detail. Let $\Theta$ be the curvature form of the renormalized connection defined via $\r$. Then we have
\begin{equation}\label{CGB}
\int_\Omega c_{n+1}(\Theta)=\chi(\Omega)+\int_M \Pi(R_{\a\ol\b\g\ol\mu}, A_{\a\b})\, \theta\wedge(d\theta)^n,
\end{equation}
where $\Pi$ is a linear combination of the complete contractions of polynomials in the Tanaka--Webster curvature and torsion without covariant derivatives. In the case of the approximate Cheng--Yau metric, $c_{n+1}(\Theta)$ agrees with $c_{n+1}(W)$ only near the boundary, so it depends on the choice of $\r$ or $\th$; see Proposition \ref{Theta-W}. Nevertheless, it is shown that the boundary integral in \eqref{CGB} is independent of the choice of $\th$ and gives a CR invariant of $M$.

In this paper, we show that the Burns--Epstein invariant given by \eqref{CGB} can be further decomposed into the sum of global CR invariants. We fix a Fefferman defining function $\r$ associated with a pseudo-Einstein contact form $\th$, and set $\omega:=-i\pa\ol\pa\log\r$. Then we can expand $c_{n+1}(\Theta)$ as 
\begin{equation}\label{chern-exp}
c_{n+1}(\Theta)=c_{n+1}(\Psi)+\sum_{m=0}^{n}\omega^{n+1-m}\wedge\Phi_m(\Theta),
\end{equation}
where $\Phi_m$ is an invariant polynomial of degree $m$. Our main theorem states that the finite part of the integral of each term in this expansion gives a CR invariant:
\begin{thm}\label{integral-inv}
Let $\Phi$ be an Ad-invariant homogeneous polynomial of degree $m\ (0\le m\le n)$ on $\mathfrak{gl}(n+1, \mathbb{C})$. Let $\r$ be the Fefferman defining function associated with a pseudo-Einstein contact form $\th$. Then we have 
\begin{equation}\label{finite-part}
{\rm fp} \int_{\rho>\e} \omega^{n+1-m}\wedge\Phi(\Theta)
=\int_M \mathcal{F}_\Phi (R_{\a\ol\b\g\ol\mu}, A_{\a\b}, \nabla_\a)\,\theta\wedge(d\theta)^n,
\end{equation}
where $\mathcal{F}_\Phi$ is a linear combination of the complete contractions of polynomials in the Tanaka--Webster curvature, torsion and their covariant derivatives. Moreover, the integral is independent of the choice of $\th$ and gives a CR invariant of $M$.
\end{thm}
A usual characteristic form for $g$ can be written by a combination of powers of $\omega$ and renormalized characteristic forms; see \S\ref{ren-chara}. Hence Theorem \ref{integral-inv} implies that the finite part
\[
{\rm fp} \int_{\rho>\e} \omega^{n+1-m}\wedge\Phi(\Psi)
\]
also gives a CR invariant of $M$. Moreover, it follows from \eqref{CGB} and \eqref{chern-exp} that the finite part of the usual Gauss--Bonnet integral
\[
{\rm fp}\int_{\r>\e} c_{n+1}(\Psi)
\]
is the sum of $\chi(\Omega)$ and a CR invariant of $M$.

When $\Phi=1\ (m=0)$, the integrand of the left-hand side of \eqref{finite-part} is equal to the volume form of $g$ up to a compactly supported form determined by the difference between $\omega$ and the K\"ahler form of $g$. The volume renormalization of the Cheng--Yau metric with respect to a Fefferman defining function is considered in \cite{HMM}, and it is shown that 
\[
{\rm fp}\int_{\r>\e}\omega^{n+1}=k_n\int_M Q',
\]
where $Q'$ is the $Q'$-curvature, introduced by Case--Yang \cite{CY} for $n=1$ and generalized by Hirachi \cite{H} to higher dimensions. Therefore, our global CR invariants in Theorem \ref{integral-inv} can be considered as generalizations of the total $Q'$-curvature.

The $Q'$-curvature is a pseudo-hermitian invariant defined for each pseudo-Einstein contact form, and the CR invariance of its integral follows from the transformation formula under rescaling, which is described in terms of the CR GJMS operator and the $P'$-operator. Thus, it is natural to expect that the CR invariant given by Theorem \ref{integral-inv} is also the integral of a ``$Q'$-like'' curvature on $M$ whose transformation formula explains the CR invariance. In this paper, we construct such a curvature quantity $\mathcal{I}'_{\Phi}$ for each invariant polynomial $\Phi$ of degree $m=n$, which gives a generalization of the $\mathcal{I}'$-curvature of Case--Gover \cite{CG}.

In \cite{CG}, they consider a local CR invariant $\mathcal{I}\in\calE(-3, -3)$ which is an analogue of the Fefferman--Graham invariant in conformal geometry. When the CR manifold $M$ is 5-dimensional, it integrates to a global CR invariant, which is though equal to $0$. They introduced an alternative ``primed'' pseudo-hermitian invariant $\mathcal{I}'$ and showed that it integrates to a non-trivial CR invariant which equals the difference between the Burns--Epstein invariant and the total $Q'$-curvature. The transformation formula of $\mathcal{I}'$ is described by a CR invariant tensor $X_\a$ such that $(X_\a\th^\a+X_{\ol\a}\th^{\ol\a})\wedge\th\wedge d\th$ gives a representative of the second Chern class 
$c_2(T^{1, 0}M)\in H^4(M; \R)$ when $\th$ is pseudo-Einstein. It then follows from the vanishing of this cohomology class that the integral of $\mathcal{I}'$ is independent of the choice of pseudo-Einstein contact form. 

By using CR tractor calculus, we generalize $\mathcal{I}'$ to $\mathcal{I}'_\Phi$ on $(2n+1)$-dimensional CR manifolds for any invariant polynomial $\Phi$ of degree $n$ on $\mathfrak{gl}(n, \mathbb{C})$. This can be considered as a primed analogue of a local CR invariant $\mathcal{I}_\Phi\in\calE(-n-1, -n-1)$. For a rescaling $\wh\th=e^\U\th$ of a pseudo-Einstein contact form, it satisfies
\[
\wh {\mathcal{I}}'_\Phi=\mathcal{I}'_\Phi - X^\Phi_\a \U^\a-X^\Phi_{\ol\a}\U^{\ol\a}
\]
with a CR invariant tensor $X^{\Phi}_\a$ such that $n^2(X^{\Phi}_\a\th^\a+X^{\Phi}_{\ol\a}\th^{\ol\a})\wedge\th\wedge (d\th)^{n-1}$ is a representative of $\Phi(T^{1, 0}M)\in H^{2n}(M; \R)$. Using the vanishing of this characteristic class, due to Takeuchi \cite{T}, 
we can prove that the total $\mathcal{I}'_\Phi$-curvature gives a CR invariant. Then we prove that two global CR invariants associated with $\Phi$ of degree $n$ coincide up to a constant multiple (Theorem \ref{fp-I-prime}):
\[
{\rm fp}\int_{\rho>\e} \omega\wedge\Phi(\Theta)=-n\int_M \mathcal{I}'_\Phi.
\]
Here, we regard $\Phi$ as an Ad-invariant polynomial on both $\mathfrak{gl}(n+1, \mathbb{C})$ and $\mathfrak{gl}(n, \mathbb{C})$ by representing it as a polynomial in $T_p(A):={\rm tr}(iA)^p$.  
To prove this equality, we relate the renormalized characteristic form to the ambient metric and establish an explicit correspondence between the ambient metric and the CR tractor calculus in the Graham--Lee setting. 

As an example, we compute the total $\mathcal{I}'_\Phi$-curvatures for a circle bundle over a complete intersection $Y$ in the complex projective space. It is given as a polynomial in the degrees of the defining polynomials of $Y$, and it shows that the invariants are non-trivial and independent of each other if we consider $\Phi$ modulo the first Chern form $c_1$. As for the other CR invariants given by Theorem \ref{integral-inv}, Takeuchi \cite{T2} computed them for Sasakian $\eta$-Einstein CR manifolds. His computation assures that the CR invariants in Theorem \ref{integral-inv} are non-trivial if $\Phi$ is not $0$ modulo $c_1$; see Remark \ref{general-case}. 
 
After this work was completed, the author was informed by Jeffrey Case and Yuya Takeuchi that they had also constructed generalizations of the $\mathcal{I}'$-curvature by a different method; they first generalize the CR invariant tensor $X_\a$ and construct the corresponding $\mathcal{I}'$-type curvature. The curvatures thus obtained are exactly same as ours. The detail of the construction and the relation to Hirachi's conjecture will appear in their paper \cite{CT}.
\bigskip

This paper is organized as follows: In \S\ref{CR-geometry}, we review some materials in CR geometry such as the Tanaka--Webster connection and the CR tractor calculus. In \S\ref{spc}, we define the Cheng--Yau metric and the ambient metric via Fefferman's approximate solution to the Monge--Amp\`ere equation, and recall the Graham--Lee connection, which plays an important role in relating the Cheng--Yau metric to CR geometry on the boundary. In \S\ref{renormalized-chara}, we introduce the renormalized connection for the Cheng--Yau metric and the renormalized characteristic forms by following Burns--Epstein. Its relation to the ambient metric and the Graham--Lee connection will be discussed. Then we construct the CR invariants associated with each invariant polynomial $\Phi$ and prove Theorem \ref{integral-inv}. We also deal with the case of the exact Cheng--Yau metric. The $\mathcal{I}'_\Phi$-curvatures are defined in \S\ref{I-prime-curv}. We clarify its relation to characteristic classes on the CR manifold and prove the CR invariance of the integral. Then, in \S\ref{relation}, we show that the total $\mathcal{I}'_\Phi$-curvature agrees with the CR invariant constructed in Theorem \ref{integral-inv}; the most part will be devoted to the preparation of the proof, where we establish the correspondence between the ambient construction and the CR tractor calculus in the Graham--Lee setting. Finally, we compute the total $\mathcal{I}'_\Phi$-curvatures for a circle bundle over a compact K\"ahler--Einstein manifold.
\bigskip

{\it Notations}: Throughout the paper, we denote the space of sections of a vector bundle by the same symbol as the bundle itself, and we adopt Einstein's summation convention. The symmetrization and the skew symmetrization  of a tensor are represented by $(, )$ and $[, ]$ respectively, and they are performed over barred indices and unbarred indices separately, e.g., 
\[
B_{[\a_1\ol\b_1\a_2\ol\b_2]}:=\frac{1}{2!2!}
(B_{\a_1\ol\b_1\a_2\ol\b_2}-B_{\a_2\ol\b_1\a_1\ol\b_2}-B_{\a_1\ol\b_2\a_2\ol\b_1}+B_{\a_2\ol\b_2\a_1\ol\b_1}).
\]
The trace-free part of $B_{\a\ol\b}$ is denoted by $B_{(\a\ol\b)_0}$. 

When a function $I(\e)$ admits an asymptotic expansion
\[
I(\e)=\sum_{m=0}^{k}\sum_{j=0}^{l_m}a_{m, j}\e^{-m}(\log\e)^j+o(1)
\]
as $\e\to0$ with some constants $a_{m, j}$, we define the logarithmic part and the finite part by 
\[
{\rm lp}\,I(\e):=a_{0, 1}, \quad {\rm fp}\,I(\e):=a_{0, 0}.
\]

\bigskip\noindent {\bf Acknowledgment} The author is grateful to Jih-Hsin Cheng and Kengo Hirachi for invaluable comments and discussions. He thanks Yuya Takeuchi for informing the author of his computation of the CR invariants for Sasakian $\eta$-Einstein manifolds. He also thanks Jeffrey Case and Yuya Takeuchi for the information of their independent construction of generalizations of the $\mathcal{I}'$-curvature and sharing their notes.

%%%%%%%%%%%%%%%%%%%%%%%%%%%%%%%%%%%%%%%%%%%%%%%%%%%%%%%%%%%%%%%%
\section{CR geometry}\label{CR-geometry}
\subsection{CR structures}
Let $M$ be a $C^\infty$-manifold of dimension $2n+1\ge 3$. An {\it almost CR structure} on $M$ is a rank-$2n$ distribution $H\subset TM$ endowed with an almost complex structure $J\in End(H)$. An almost CR structure $(H, J)$ is called a {\it CR structure} when it satisfies the integrability condition $[T^{1, 0}M, T^{1, 0}M]\subset T^{1, 0}M$, where $\C H=T^{1, 0}M\oplus T^{0, 1}M$ is the eigenspace decomposition for $J$.  We assume that there exists a global real 1-form $\theta$ whose kernel is $H$, and define the {\it Levi form} by 
\[
h_\theta(X, Y):=d\theta(X, JY), \quad X, Y\in H.
\]  
The Levi form is a $J$-invariant symmetric form and we say the CR structure is {\it strictly pseudoconvex} if we can take $\theta$ for which $h_\theta$ is positive definite. Hereafter, we  assume that $(H, J)$ is a strictly pseudoconvex CR structure, and call $\theta$ a {\it contact form} since in this case $H$ is a contact distribution. A contact form is determined up to multiplications of positive functions. We define the orientation of $M$ so that $\theta\wedge (d\theta)^n>0$. 

The {\it CR canonical bundle} is the complex line bundle defined by 
\[
K_M:=\{\zeta\in \wedge^{n+1}\C T^* M\ |\ \ol Z\lrcorner\,\zeta=0\ \mbox{for all}\ Z\in T^{1, 0}M\}.
\]
We fix a complex line bundle $\calE(1, 0)$ which satisfies $\calE(1, 0)^{\otimes(-n-2)}=K_M$ and define the {\it CR density of weight} $(w, w')\ (w-w'\in \mathbb{Z})$ by   
\[
\calE(w, w'):=\calE(1, 0)^{\otimes w}\otimes \ol{\calE(1, 0)}{}^{\otimes w'}.
\]
As in the case of spin structures in Riemannian geometry, $\calE(1, 0)$ may not exist globally, but it always exists locally. When $M$ is a hypersurface in $\C^{n+1}$, $K_M$ is trivial and $\calE(1, 0)$ exists globally. We note that $\calE(w, w)$ can always be defined globally. Moreover, when $w\in \R$ it has a reduction to an $\R_+$-bundle. We call a positive section $0<\tau\in\calE(1, 1)$ a {\it CR scale}. If $\sigma\in\calE(1, 0)$ is a nonvanishing section, then $|\sigma|^2=\sigma\bar\sigma$ defines a CR scale. We also call $\sigma$ itself a CR scale.

We denote $T^{1, 0}M, T^{0, 1}M$ and their duals by $\calE^{\a}, \calE^{\ol\a}, \calE_{\a}, \calE_{\ol\a}$ respectively, and use similar notations for their tensor products, e.g., 
$\calE_\a\otimes\calE_{\ol\b}=\calE_{\a\ol\b}$. We also put $(w, w')$ to these bundles to represent the tensor product with $\calE(w, w')$, e.g., $\calE_\a\otimes\calE(w, w')=\calE_\a(w, w')$.

Let $\zeta$ be a nowhere vanishing section of $K_M$. A contact form $\theta$ is {\it volume normalized} by $\zeta$ if it satisfies
\[
\theta\wedge(d\theta)^n=i^{n^2}n! \theta\wedge(T\lrcorner\, \zeta)\wedge(T\lrcorner\, \ol\zeta),
\] 
where $T$ is the Reeb vector field: $\theta(T)=1, T\lrcorner\, d\theta=0$. Such a $\theta$ is unique, and the weighted 1-form
\[
\bth:=\theta\otimes|\zeta|^{-2/(n+2)}\in T^*M\otimes\calE(1, 1)
\]
is independent of choice of $\zeta$. We associate a CR scale $\tau\in\calE(1, 1)$ with the contact form $\theta=\tau^{-1}\bth$; this gives a one-to-one correspondence between choices of a CR scale and those of a contact form.
%%%%%%%%%%%%%%%%%%%%%%%%%%%%%%%%%%%%%%%%%%%%%%%%%%%%%%%%%%%%%%%%%%%%%%%
\subsection{The Tanaka--Webster connection} For a choice of contact form $\theta$, one can define a canonical linear connection on $TM$, called the {\it Tanaka--Webster connection}.

We take a local frame $\{Z_\a\}$ of $T^{1, 0}M$ and consider the local frame 
$\{T, Z_\a, Z_{\ol\a}:=\ol{Z_\a}\}$ of $\C TM$, which is called an {\it admissible frame}.
Let $\{\theta, \theta^\a, \theta^{\ol\a}\}$ be the dual frame. Then we have 
\begin{equation}\label{d-th}
d\theta=i h_{\a\ol\b}\theta^\a\wedge\theta^{\ol\b},
\end{equation}
where $h_{\a\ol\b}=h_\th(Z_\a, Z_{\ol\b})$ is the components of the Levi form. By using the CR scale $\tau\in\calE(1, 1)$ corresponding to $\th$, we define the weighted Levi form 
\[
\bh_{\a\ol\b}:=\tau h_{\a\ol\b}\in\calE_{\a\ol\b}(1, 1),
\]
which is invariant under rescaling of $\th$. We lower and raise the indices by $\bh_{\a\ol\b}$ and its inverse $\bh^{\a\ol\b}\in\calE^{\a\ol\b}(-1, -1)$. 

The Tanaka--Webster connection $\nabla$ is given by 
\[
\nabla T=0, \quad \nabla Z_\a=\omega_\a{}^\b \otimes Z_\b, \quad \nabla Z_{\ol\a}=\omega_{\ol\a}{}^{\ol\b} \otimes Z_{\ol\b},
\]
where $\omega_{\ol\a}{}^{\ol\b}:=\ol{\omega_\a{}^\b}$ and $\omega_\a{}^\b$ are characterized by the structure equations:
\begin{align*}
d\th^\b&=\th^\a\wedge\omega_\a{}^\b+A^\b{}_{\ol\g}\bth\wedge\th^{\ol\g}, \\
d h_{\a\ol\b}&=\omega_\a{}^\g h_{\g\ol\b}+h_{\a\ol\g}\omega_{\ol\b}{}^{\ol\g}.
\end{align*}
In particular, $\nabla$ preserves the Levi form: $\nabla h_{\a\ol\b}=0$. The tensor $A^\b{}_{\ol\g}\in\calE^\b{}_{\ol\g}(-1, -1)$ is called the {\it Tanaka--Webster torsion}. 
By differentiating \eqref{d-th}, one can see that $A_{\a\b}:=\ol{A_{\ol\a\ol\b}}\in\calE_{\a\b}$ is a symmetric tensor. 

The Tanaka--Webster connection also defines a connection on the CR density bundle $\mathcal{E}(w, w')$, whose connection form is given by 
\[
\frac{1}{n+2}(w\, \omega_\a{}^\a+w'\omega_{\ol\a}{}^{\ol\a}).
\]
Since $\nabla$ preserves $T^{1, 0}M$, covariant differentiations on tensor bundles such as 
$\calE_{\a\ol\b\g}(w, w')$ make sense and we denote the components of the covariant derivative as 
\[
\nabla_\mu f_{\a\ol\b\g}=f_{\a\ol\b\g, \mu}, \quad \nabla_{\ol\mu} f_{\a\ol\b\g}=f_{\a\ol\b\g, \ol\mu}, \quad \nabla_T f_{\a\ol\b\g}=f_{\a\ol\b\g, T}.
\]
When we differentiate a tensor in the $T$-direction, we lower the weight by $\tau^{-1}$ so that $\nabla_Tf_{\a\ol\b\g}\in\calE_{\a\ol\b\g}(w-1, w'-1)$. 

If we take $\{Z_\a\}$ such that $h_{\a\ol\b}=\d_{\a\ol\b}$, then $\th$ is volume normalized by 
\[
\zeta:=\th\wedge\th^1\wedge\cdots\wedge\th^n.
\]
By $\nabla\th=0$, $\nabla\zeta=-\omega_\a{}^\a\otimes\zeta$ and $\mbox{Re}\,\omega_\a{}^\a=0$, we have $\nabla\bth=0$, which implies 
\[
\nabla\tau=0, \quad  \nabla\bh_{\a\ol\b}=0.
\]
The curvature form $d\omega_\a{}^\b-\omega_\a{}^\g\wedge\omega_\g{}^\b$ is expressed as 
\[
R_\a{}^\b{}_{\g\overline\mu} \th^\g\wedge\th^{\overline \mu}+A_{\a\g,}{}^\b\th^\g\wedge\bth-A^\b{}_{\overline \g,\a}\th^{\overline \g}\wedge\bth 
  -iA_{\a\g}\th^\g\wedge\th^\b+i\bh_{\a\overline\g}A^\b{}_{\overline\mu}
\th^{\overline\g}\wedge\th^{\overline\mu}.
\]
The tensor $R_\a{}^\b{}_{\g\ol\mu}$ is called the {\it Tanaka--Webster curvature}. By taking traces, we obtain the {\it Tanaka--Webster Ricci tensor} and the {\it Tanaka--Webster scalar curvature}: 
\[
R_{\g\ol\mu}:=R_\a{}^\a{}_{\g\ol\mu}\in\calE_{\g\ol\mu}, \quad R:=R_\g{}^\g\in\calE(-1, -1).
\]
Under a rescaling of the contact form $\wh\th=e^\U\th$, these curvature quantities satisfy transformation formulas involving derivatives of the scaling factor $\U$; see \cite{Lee}.
%%%%%%%%%%%%%%%%%%%%%%%%%%%%%%%%%%%%%%%%%%%%%%%%%%%%%%%%%%%%%%%%%%%%
\subsection{Pseudo-Einstein contact forms}
We will review the notion of pseudo-Einstein contact form. We refer the reader to 
\cite{Lee, H, HMM} for details.

A contact form $\theta$ is called {\it pseudo-Einstein} if it is locally volume normalized by a closed section $\zeta\in K_M$. In terms of the Tanaka--Webster connection, it is characterized by the equation
\[
R_{\a\ol\b}=\frac{1}{n}R\bh_{\a\ol\b}\ \,(n\ge 2), \quad \nabla_\a R=i A_{\a\b,}{}^\b\ \,(n=1).
\]
A CR scale $\sigma\in\calE(1, 0)$, or $\tau=|\sigma|^2\in\calE(1, 1)$, is called a {\it pseudo-Einstein CR scale} when the corresponding contact form $\tau^{-1}\bth$ is pseudo-Einstein. 

A complex function $f$ on $M$ is called a {\it CR function} if $\ol Zf=0$ for all $Z\in T^{1, 0}M$, and a real function $\U$ is called a {\it CR pluriharmonic function} if it is locally the real part of a CR function. It is known that when $\th$ is pseudo-Einstein, a contact form $ e^{\U}\th$ is also pseudo-Einstein if and only if $\U$ is CR pluriharmonic. The CR plurihamonicity has the following characterization. Let $\Delta_b=-(\nabla_\a\nabla^\a+\nabla^\a\nabla_\a)$ be the {\it sublaplacian}, and set
\[
\eta:=i(\U_\a\theta^\a-\U_{\ol\a}\theta^{\ol\a})-\frac{1}{n}(\Delta_b \U) \theta.
\]
\begin{prop}[{\cite[Lemma 3.2]{H}}]\label{cr-pluri}
$\U$ is CR pluriharmonic if and only if $d\eta=0$.
\end{prop}
 %%%%%%%%%%%%%%%%%%%%%%%%%%%%%%%%%%%%%%%%%%%%%%%%%%%%%%%%%%%%%%%%%%%%%
\subsection{CR tractor bundles}
The CR tractor bundles are vector bundles associated to the CR Cartan bundle, and they are endowed with canonical linear connections induced by the Cartan connection. Here we define them in a direct way by following Gover--Graham \cite{GG}; see also \cite{CG}. The relation to the ambient metric construction will be discussed in \S\ref{ambient-tractor}.

If we fix a contact form $\th$, the {\it (standard) CR tractor bundle} $\calE^A$ is identified with the direct sum:
\[
\calE^A\overset{\th}{\cong}\begin{matrix}
\mathcal{E}(0,1)\\
\oplus \\
\mathcal{E}^\a (-1,0) \\
\oplus \\
\mathcal{E}(-1,0)
\end{matrix}.
\]
An element $\nu^A={}^t(\sigma, \nu^\a, \lambda)$ is called a {\it tractor}. The expression 
${}^t(\hat\sigma, \hat\nu^\a, \hat\lambda)$ for the same tractor in terms of another contact form $\wh\th=e^{\U}\th$ is related to the original one by
\[
\begin{pmatrix}
\hat\sigma \\
\hat\nu^\a \\
\hat\lambda
\end{pmatrix}
=
\begin{pmatrix}
\sigma \\
\nu^\a+\U^\a\sigma \\
\lambda-\U_\b \nu^\b-\frac{1}{2}(\U_\beta\U^\beta-i\U_T)\sigma
\end{pmatrix}.
\]
For a choice of $\th$, we define
\[
Z^A=
\begin{pmatrix} 0 \\ 0 \\ 1\end{pmatrix}\in\mathcal{E}^A(1,0), \quad 
W^A_\b=
\begin{pmatrix} 0 \\ \d_\b{}^\a \\ 0 \end{pmatrix}\in\mathcal{E}_\b^A(1,0),\quad 
Y^A=
\begin{pmatrix} 1 \\ 0 \\ 0 \end{pmatrix}\in\mathcal{E}^A(0,-1),
\]
where $\calE_\b^A$ denotes the tensor product $\calE_\b\otimes\calE^A$. Then any tractor is represented as $\sigma Y^A+\nu^\a W_\a^A+\lambda Z^A$. For a rescaling $\wh\th=e^\U\th$, we have
\begin{equation}\label{Z-W-Y}
\begin{gathered}
\widehat Z^A=Z^A, \quad \widehat W^A_\a=W^A_\a+\U_\a Z^A, \\
\widehat Y^A =Y^A-\U^\a W^A_\a-\frac{1}{2}(\U_\b\U^\b+i\U_T)Z^A.
\end{gathered}
\end{equation}

The dual bundle of $\calE^A$ is denoted by $\calE_A$; the expression as a direct sum and the transformation formula of components are given by
\[
\calE_A\overset{\th}{\cong} 
\begin{matrix}
\mathcal{E}(1,0)\\
\oplus \\
\mathcal{E}_\a (1,0) \\
\oplus \\
\mathcal{E}(0,-1)
\end{matrix}, \quad\quad 
\begin{pmatrix}
\hat\sigma \\
\hat\nu_\a \\
\hat\lambda
\end{pmatrix}
=
\begin{pmatrix}
\sigma \\
\nu_\a+\U_\a\sigma \\
\lambda-\U^\b \nu_\b-\frac{1}{2}(\U_\beta\U^\beta+i\U_T)\sigma
\end{pmatrix}.
\]
We denote by $\calE^{\ol A}, \calE_{\ol A}$ the complex conjugate bundle of $\calE^A, \calE_A$ and use the index notation for tensor products, e.g., $\calE_{A\ol B}=\calE_A\otimes\calE_{\ol B}$. These bundles are also called CR tractor bundles.

The {\it tractor metric} is the Lorentzian hermitian metric on $\calE^A$ defined by
\[
h_{A\ol B}\nu^A\ol{\nu'^B}:=\sigma\ol{\lambda'}+\lambda\ol{\sigma'}+\bh_{\a\ol\b}\nu^\a\ol{\nu'^\b},
\]
which is independent of choice of $\th$. We lower and raise tractor indices by $h_{A\ol B}\in\calE_{A\ol B}$ and its inverse $h^{A\ol B}\in\calE^{A\ol B}$.

There exists a CR invariant linear connection on each tractor bundle, called the {\it tractor connection}. In terms of the Tanaka--Webster connection, it is defined by
\begin{equation}\label{tractor-conn}
\begin{gathered}
\nabla_\b v^A=
\begin{pmatrix}
\nabla_\b \sigma \\
\nabla_\b\nu^\a+\d_\b{}^\a\lambda+P_{\b}{}^\a\sigma \\
\nabla_\b\lambda-iA_{\b\g}\nu^\g-T_\b\sigma
\end{pmatrix},
\quad
\nabla_{\ol\b} v^A=
\begin{pmatrix}
\nabla_{\ol\b}\sigma-\nu_{\ol\b}\\
\nabla_{\ol\b}\nu^\a-iA_{\ol\b}{}^\a\sigma \\
\nabla_{\ol\b}\lambda-P_{\g\ol\b}\nu^{\g}+T_{\ol\b}\sigma
\end{pmatrix}, \\
\nabla_T v^A=
\begin{pmatrix}
\nabla_T\sigma-\frac{i}{n+2}P\sigma+i\lambda \\
\nabla_T\nu^\a+iP_{\b}{}^\a\nu^\b-\frac{i}{n+2}P\nu^\a-2iT^\a\sigma \\
\nabla_T\lambda-\frac{i}{n+2}P\lambda-2iT_\a \nu^\a-iS\sigma
\end{pmatrix}.\qquad\qquad\qquad
\end{gathered}
\end{equation}
The Tanaka--Webster curvature quantities which appear in the formulas are defined as follows:
\begin{align*}
P_{\a\ol\b}&:=\frac{1}{n+2}\bigl(R_{\a\ol\b}-\frac{{R}}{2(n+1)}\bh_{\a\ol\b}\bigr), &
P&:=P_\a{}^\a=\frac{R}{2(n+1)}, \\
T_\a&=\ol{T_{\ol\a}}:=\frac{1}{n+2}(\nabla_\a P-i A_{\a\g,}{}^\g), &
S&:=-\frac{1}{n}(T_{\a,}{}^\a+T^\a{}_{,\a}+|P_{\a\ol\b}|^2-|A_{\a\b}|^2).
\end{align*}
The tractor connection preserves the tractor metric: 
$\nabla h_{A\ol B}=0$. The curvature form of the tractor connection is given by
\begin{align*}
\Omega_A{}^B&=
(S_\a{}^\b{}_{\g\ol\mu}\th^\g\wedge\th^{\ol\mu}+V_\g{}^\b{}_\a\th^\g\wedge\bth-
V^\b{}_{\a\ol\g}\th^{\ol\g}\wedge\bth)\,W_A^\a W_\b^B \\
&\quad +(iV_{\g\ol\mu\a}\th^\g\wedge\th^{\ol\mu}+Q_{\a\g}\th^\g\wedge\bth-iU_{\a\ol\g}\th^{\ol\g}\wedge\bth)\,W_A^\a Z^B \\
&\quad -(iV_{\ol\mu\g}{}^\b\th^\g\wedge\th^{\ol\mu}+iU_\g{}^\b\th^\g\wedge\bth+Q_{\ol\g}{}^\b\th^{\ol\g}\wedge\bth)\,Z_AW_\b^B \\
&\quad +(U_{\g\ol\mu}\th^\g\wedge\th^{\ol\mu}+Y_\g\th^\g\wedge\bth-Y_{\ol\g}\th^{\ol\g}\wedge\bth)\,Z_A Z^B,
\end{align*}
where 
\begin{align*}
S_{\a\ol\b\g\ol\mu}&:=R_{\a\ol\b\g\ol\mu}-P_{\a\ol\b}\bh_{\g\ol\mu}-P_{\g\ol\b}\bh_{\a\ol\mu}-P_{\g\ol\mu}\bh_{\a\ol\b}-P_{\a\ol\mu}\bh_{\g\ol\b}, \\
V_{\a\ol\b\g}&:=\nabla_{\ol\b}A_{\a\g}+i\nabla_\g P_{\a\ol\b}-iT_\g \bh_{\a\ol\b}-2iT_\a \bh_{\g\ol\b}, \\
Q_{\a\b}&:=i\nabla_T A_{\a\b}-2i\nabla_\b T_\a+2P_\a{}^\g A_{\g\b}, \\
U_{\a\ol\b}&:=\nabla_{\ol\b}T_\a+\nabla_\a T_{\ol\b}+P_\a{}^\g P_{\g\ol\b}-A_{\a\g}A^\g{}_{\ol\b}+S\bh_{\a\ol\b}, \\
Y_\a&:=\nabla_T T_\a-i\nabla_\a S+2i P_\a{}^\g T_\g-3A_{\a\g}T^\g,
\end{align*}
and $V_{\ol\a\b\ol\g}:=\ol{V_{\a\ol\b\g}},\ Q_{\ol{\a}\ol{\b}}:=\ol{Q_{\a\b}},\ Y_{\ol\a}:=\ol{Y_\a}$. The tensor $S_{\a\ol\b\g\ol\mu}$ is called the {\it Chen-Moser} tensor. It is a CR invariant tensor, and it vanishes identically if and only if $n=1$ or the CR manifold is locally isomorphic to the standard sphere.

For any weighted tractor $t\in\calE_{*}(w, w')$, we can define a CR invariant linear differential operator $t\mapsto D_A t\in \calE_{A*}(w-1, w')$, called the {\it tractor $D$-operator}.  In a scale $\th$, it is given by  
\[
D_A t :=
\begin{pmatrix}
w(n+w+w^\prime) t \\
(n+w+w^\prime)\nabla_\a t \\
-\nabla^\b\nabla_\b t-iw\nabla_T t-w\bigl(1+\frac{w^\prime-w}{n+2}\bigr)Pt
\end{pmatrix}.
\]
We also have its complex conjugate: $D_{\ol A}t:=\ol{D_{A}\ol t}\in\calE_{\ol A*}(w, w'-1)$. 

%%%%%%%%%%%%%%%%%%%%%%%%%%%%%%%%%%%%%%%%%%%%%%%%%%%%%%%%%%%%%%%%%%%%%%%%%%%%
\section{Strictly pseudoconvex domains} \label{spc}
\subsection{The Monge--Amp\`ere equation}
Let $X$ be a complex manifold of complex dimension $n+1\ge 2$. A relatively compact domain $\Omega\subset X$ with the smooth boundary $M=\partial\Omega$ is called {\it strictly pseudoconvex} if the induced CR structure on $M$ is strictly pseudoconvex. Let $K_X$ be the canonical bundle of $X$. We define the {\it ambient space} of the CR manifold $M$ by the $\C^*$-bundle $\wt X:=K_X\setminus\{0\}$ over $X$. We also define $\mathcal{N}:=K_M\setminus\{0\}$, which is isomorphic to the restriction $\wt X|_M$.
A section of the complex line bundle 
\[
\wt\calE(w):=(K_X\otimes \ol{K_X})^{-w/(n+2)}
\]
is called a {\it density of weight} $w$. The restriction $\wt\calE(w)|_M$ is canonically isomorphic to the CR density $\calE(w, w)$. A density $f\in\wt\calE(w)$ can be identified with a function on $\wt X$ which satisfies the homogeneity condition $\delta_\lambda^* f=|\lambda|^{2w}f\ (\lambda\in\C^*)$ with respect to the dilation $\d_\lambda(u):=\lambda^{n+2}u$. We call a real density $\br\in\wt\calE(1)$ a {\it defining density} if the corresponding homogeneous function on $\wt X$ gives a defining function of $\mathcal{N}$ which is positive on the fibers over $\Omega$: $\mathcal{N}=\{\br=0\},\ d\br|_{\mathcal{N}}\neq0$.
By the {\it pseudoconvex side}, we mean the side on which $\br>0$.

If we take local holomorphic coordinates $(z^1, \dots, z^{n+1})$, we have the local density
\[
\wt\tau=|dz^1\wedge\cdots\wedge dz^{n+1}|^{-2/(n+2)}\in\wt\calE(1).
\]
Using this density, we define the {\it Monge--Amp\`ere operator} acting on a real density $\br\in\wt\calE(1)$ by
\[
\mathcal{J}[\br]:=(-1)^{n+1}\det
\begin{pmatrix}
\r & \partial_{z^i}\r \\
\partial_{\bar z^j}\r & \partial_{z^i}\partial_{\bar z^j} \r
\end{pmatrix},
\]
where $\rho:={\wt\tau}^{-1}\br\in\wt\calE(0)$. It can be verified that $\mathcal{J}$ is independent of the choice of coordinates $(z^i)$. Fefferman constructed an approximate solution to the Monge--Amp\`ere equation $\mathcal{J}[\br]=1$:
\begin{thm}[\cite{F}]
There exists a defining density $\br\in\wt\calE(1)$ which satisfies
\begin{equation}\label{MA}
\mathcal{J}[\br]=1+O(r^{n+2}),
\end{equation}
where $r\in C^{\infty}(X)$ is an arbitrary defining function of $\Omega$. Moreover, such a density is unique modulo $O(r^{n+3})$.
\end{thm}
We fix a density $\br\in\wt\calE(1)$ given in the theorem above and call it a {\it Fefferman defining density}. 
Let $\wt\tau\in\wt\calE(1)$ be a density which restricts to a CR scale $\tau=\wt\tau|_M\in\calE(1, 1)$. Then $\r :=\wt\tau^{-1}\br$ becomes a defining function of $\Omega$ near $M$, and the 1-form
\[
\frac{i}{2}(\pa\r-\ol\pa\r)|_{TM}
\]
coincides with the contact form $\th=\tau^{-1}\bth$; see {\cite[Proposition 5.2]{Fa}}. The pseudo-Einstein condition on $\th$ can be characterized in terms of pluriharmonic extension of $\log\tau$:
\begin{prop}[{\cite[Proposition 2.6]{HMM}}]\label{p-E}
A CR scale $\tau\in\calE(1, 1)$ is pseudo-Einstein if and only if it has an extension $\wt\tau\in\wt\calE(1)$ which satisfies $\pa\ol\pa\log\wt\tau=0$ on the pseudoconvex side of a neighborhood of $\mathcal{N}$. Such an extension is unique near $\mathcal{N}$ on the pseudoconvex side.
\end{prop}
By this proposition, we also have
\begin{prop}\label{pluriharmonic-ext}
$\U\in C^\infty(M)$ is CR pluriharmonic if and only if it has an extension $\wt\U\in C^\infty(\ol\Omega)$ such that $\pa\ol\pa\wt\U=0$ near $M$.
\end{prop}

%%%%%%%%%%%%%%%%%%%%%%%%%%%%%%%%%%%%%%%%%%%%%%%%%%%%%%%%%%%%%%%%%%%%%
\subsection{The ambient metric and the Cheng--Yau metric}\label{ambient-CY}
Let $\br\in\wt\calE(1)$ be a Fefferman defining density. We define the {\it ambient metric} on $\wt X$ near $\mathcal{N}$ by the Lorentz--K\"ahler metric $\wt g$ with the K\"ahler form $-i\pa\ol\pa\br$. By the Monge--Amp\`ere equation \eqref{MA}, the Ricci form of $\wt g$ satisfies
\begin{equation}\label{ricci-flat}
\Ric (\wt g)=\pa\ol\pa \,O(r^{n+2}).
\end{equation}

The $(1, 1)$-form $-i\pa\ol\pa\log\br$ descends to a K\"ahler form on $\Omega$ near $M$. We extend this to a hermitian metric $g$ on $\Omega$ and call it the {\it Cheng--Yau metric}. By \eqref{MA}, the Cheng--Yau metric satisfies the following approximate Einstein equation:
\begin{equation*}
\Ric(g)+(n+2)g=\pa\ol\pa\, O(r^{n+2}).
\end{equation*}
When $X=\C^{n+1}$, we can use the exact solution to the Monge--Amp\`ere equation constructed by Cheng--Yau \cite{ChY}. In this case, the metric satisfies the Einstein equation exactly. We will deal with the exact Cheng--Yau metric in \S \ref{exact}.

Suppose that $\th$ is a pseudo-Einstein contact form and let $\tau\in\calE(1, 1)$ be the corresponding CR scale. Then, by Proposition \ref{p-E}, we have an extension $\wt\tau\in\wt\calE(1)$ with $\pa\ol\pa\log\wt\tau=0$ near $\mathcal{N}$. Hence we obtain a defining function $\r:=\wt\tau^{-1}\br$ such that 
\[
\frac{i}{2}(\pa\r-\ol\pa\r)|_{TM}=\th, \quad g=-i\pa\ol\pa\log\r\ \ {\rm near}\ M,
\]
which is unique near $M$. We call $\r$ a {\it defining function associated with the pseudo-Einstein contact form} $\th$. Conversely, if $\r$ is a defining function such that $g=-i\pa\ol\pa\log\r$ near $M$, then the contact form $\th:=(i/2)(\pa\r-\ol\pa\r)|_{TM}$ is pseudo-Einstein. Thus,  there is a one to one correspondence between pseudo-Einstein contact forms on $M$ and the germs of defining functions along $M$ whose logarithms give K\"ahler potentials of the Cheng--Yau metric.  
%%%%%%%%%%%%%%%%%%%%%%%%%%%%%%%%%%%%%%%%%%%%%%%%%%%%%%%%%%%%%%%%%
\subsection{The Graham--Lee connection} 
Let $\r\in C^\infty(X)$ be an arbitrary defining function of $\Omega$ which is positive in $\Omega$. By following \cite{GL}, we introduce a linear connection on a neighborhood of $M$ in $X$ which is a natural extension of the Tanaka--Webster connection for $\th:=(i/2)(\pa\r-\ol\pa\r)|_{TM}$; see also \cite{HMM}. 

By strict pseudoconvexity, there exists a unique $(1, 0)$-vector field $\xi$ on $X$ near $M$ which satisfies
\[
\pa\r(\xi)=1, \quad -\xi\lrcorner\, \pa\ol\pa\r=\kappa\ol\pa\r
\]
for some real function $\kappa$. We call $\kappa$ the {\it transverse curvature}. Let $\{Z_\a\}$ be a local frame for ${\rm Ker}\,\pa\r\subset T^{1, 0}X$. Then, $\{Z_\a, Z_\infty:=\xi\}$ gives a local frame for $T^{1, 0}X$ called a {\it Graham--Lee frame}. Note that $\{Z_\a\}$ is a local frame for $T^{1, 0}M$ at the boundary. We write $\{\th^\a, \th^\infty:=\pa\r\}$ for the dual $(1, 0)$-coframe. We set $Z_{\ol\a}:=\ol{Z_\a}, \th^{\ol\a}:=\ol{\th^\a}$. Then
\[
-\pa\ol\pa\rho=h_{\a\ol\b}\theta^\a\wedge\theta^{\ol\b}+\kappa\pa\rho\wedge\ol\pa\rho,
\]
where $h_{\a\ol\b}:=-\pa\ol\pa\r(Z_\a, Z_{\ol\b})$ is the Levi form on each level set of $\r$ for the contact form given by the restriction of $\vartheta:=(i/2)(\pa\r-\ol\pa\r)$. We lower and raise indices by $h_{\a\ol\b}$ and its inverse $h^{\a\ol\b}$. When $\r$ is a Fefferman defining function, the Cheng--Yau metric is expressed near $M$ by
\[
-\pa\ol\pa\log\r=\frac{h_{\a\ol\b}}{\rho}\theta^\a\wedge\theta^{\ol\b}+\frac{1+\kappa\rho}{\rho^2}\pa\rho\wedge\ol\pa\rho.
\]
If we write
\[
\xi=N+\frac{i}{2}T
\]
with real vector fields $N, T$, then we have
\[
d\r(N)=\vartheta(T)=1, \quad d\r(T)=\vartheta(N)=0.
\]
Since $T\lrcorner\, d\vartheta=\kappa d\r$, $T$ is the Reeb vector field for $\vartheta$ on each level set of $\r$.

\begin{prop}[{\cite[Proposition 1.1]{GL}}]
There exists a unique linear connection $\nabla$ on $TX$ near $M$ which satisfies the following conditions:
\begin{itemize}
\item [(i)] $\nabla$ preserves ${\rm Ker}\, \pa\r$ and $\xi$ is parallel: $\nabla \xi=0$.
\item [(ii)] The connection 1-forms $\varphi_\a{}^\b$ defined by $\nabla Z_\a=\varphi_\a{}^\b\otimes Z_\b$ satisfies the structure equations
\begin{align*}
d\th^\a&=\th^\b\wedge\varphi_\b{}^\a+iA^\a{}_{\ol\b}\pa\r\wedge\th^{\ol\b}-\kappa^\a \pa\r\wedge\ol{\pa}\r-\frac{1}{2}\kappa d\r\wedge\th^\a, \\
d h_{\a\ol\b}&=\varphi_\a{}^\g h_{\g\ol\b}+h_{\a\ol\g}\,\ol{\varphi_\b{}^\g}.
\end{align*}
\end{itemize}
\end{prop}
We call $\nabla$ the {\it Graham--Lee connection}. By the structure equations above and $\nabla T=0$, we can see that $\nabla$ restricts to the Tanaka--Webster connection on each level set of $\r$. The curvature form $\Phi_\a{}^\b=d\varphi_\a{}^\b-\varphi_\a{}^\g\wedge\varphi_\g{}^\b$ is given by
\begin{equation}\label{curv-GL}
\begin{aligned}
\Phi_\a{}^\b&=
R_\a{}^\b{}_{\g\ol\mu}\th^\g\wedge\th^{\ol\mu}-iA_{\a\g, }{}^\b\th^\g\wedge\ol{\pa}\r-iA^\b{}_{\ol\g, \a}\th^{\ol\g}\wedge\pa\r \\
&\quad -iA_{\a\g}\d_{\mu}{}^\b\th^\g\wedge\th^\mu+ih_{\a\ol\g}A^\b{}_{\ol\mu}\th^{\ol\g}\wedge\th^{\ol\mu} \\
&\quad -d\r\wedge\Bigl(\kappa_\a\th^\b-\kappa^\b h_{\a\ol\g}\th^{\ol\g}+\frac{1}{2}\d_\a{}^\b(\kappa_\g\th^\g-\kappa_{\ol\g}\th^{\ol\g})\Bigr) \\
&\quad -\frac{i}{2}(\kappa_\a{}^\b+\kappa^\b{}_\a+2A_{\a\g}A^{\b\g})d\r\wedge\vartheta.
\end{aligned}
\end{equation}

The Graham--Lee connection plays an important role in relating the ambient metric or the Cheng--Yau metric to the CR geometry of the boundary. When we express CR invariants of $M$ constructed by the Cheng--Yau metric in terms of the Tanaka--Webster connection, we need the Taylor expansions of curvature quantities of the Graham--Lee connection along $M$. 
We will derive some differential equations for this purpose. We fix an arbitrary point $p_0\in\Omega$ near $M$ and take a local frame $\{Z_\a\}$ of ${\rm Ker}\,\pa\r$ such that $\varphi_\a{}^\b(p_0)=0$. We differentiate \eqref{curv-GL} at $p_0$ and look at the coefficient of $d\r\wedge\th^\g\wedge\th^{\ol\mu}$. Then we have
\begin{equation}\label{Omega-N}
\begin{aligned}
\nabla_N R_{\a\ol\b\g\ol\mu}&=
\kappa R_{\a\ol\b\g\ol\mu}-\frac{i}{2}A_{\a\g, \ol\b\ol\mu}+\frac{i}{2}A_{\ol\b\ol\mu, \a\g}-A_{\a\g}A_{\ol\b\ol\mu} \\
&\quad +\frac{1}{2}A_{\a\nu}A^\nu{}_{\ol\mu}h_{\g\ol\b}+\frac{1}{2}A_{\ol\b}{}^\nu A_{\nu\g}h_{\a\ol\mu}+\kappa_{\a\ol\mu} h_{\g\ol\b}+\kappa_{\ol\b\g}h_{\a\ol\mu} \\
&\quad +\frac{1}{2}(\kappa_{\g\ol\mu}+\kappa_{\ol\mu\g})h_{\a\ol\b}+\frac{1}{2}(\kappa_{\a\ol\b}+\kappa_{\ol\b\a}+2A_{\a\nu}A^\nu{}_{\ol\b})h_{\g\ol\mu}.
\end{aligned}
\end{equation}  
Similarly, differentiating \eqref{curv-GL} and looking at the coefficient of $\vartheta\wedge d\r\wedge \theta^{\ol\g}$, we have
\begin{equation}\label{A-N}
\nabla_N A_{\a\b}=\kappa A_{\a\b}-i\kappa_{\a\b}-\frac{i}{2}A_{\a\b, T}.
\end{equation}
We need the Monge--Amp\`ere equation to derive an equation for $\kappa$, so we assume that $\r$ is a Fefferman defining function associated with some pseudo-Einstein contact form. Then, the transverse curvature satisfies
\begin{equation}\label{kappa-M}
\kappa|_M=\frac{2}{n}P
\end{equation}
and
\begin{equation}
\begin{aligned}\label{kappa-N}
& n N\kappa-\mu\r\,\ol\xi\xi\kappa+\frac{1}{2}(\Delta_b \kappa-2|A_{\a\b}|^2)-n\kappa^2
-2\mu\r\kappa^3-\mu^2\r^2\kappa^2\ol\xi\kappa \\
&\qquad-\mu^2\r\kappa^3 +\mu^2\r^2(\xi\kappa)(\ol\xi\kappa)+\mu^2\r\kappa(\xi\kappa)+4\mu\r\kappa (N\kappa)-\mu\r\kappa_\a\kappa^\a=O(\r^n),
\end{aligned}
\end{equation}
where we have set $\mu:=(1+\kappa\r)^{-1}$; see  {\cite[(4.9), (4.15)]{Mar}}.
We differentiate the equations \eqref{Omega-N}, \eqref{A-N}, and \eqref{kappa-N} in the $N$-direction repeatedly. The commutators of $\nabla_N$ and tangential differentiations are computed with the torsion and curvature tensors of the Graham--Lee connection, which are expressed by $h_{\a\ol\b}, R_{\a\ol\b\g\ol\mu}, A_{\a\b}, \kappa$ and their tangential derivatives. As a result, we obtain the expression of lower order Taylor coefficients in terms of the Tanaka--Webster connection:  
\begin{prop}\label{exp-GL}
Let $\nabla$ be the Graham--Lee connection for a Fefferman defining function. Then 
the boundary values of 
\begin{equation}\label{terms}
\nabla_N^{p}R_{\a\ol\b\g\ol\mu}, \quad \nabla_N^{p} A_{\a\b}, \quad N^{p-1} \kappa \quad (p\le n+1)
\end{equation}
are expressed in terms of the Tanaka--Webster curvature quantities on $M$.
\end{prop}

%%%%%%%%%%%%%%%%%%%%%%%%%%%%%%%%%%%%%%%%%%%%%%%%%%%%%%%%%%%%%%%%%%%%%%%%%

\section{Renormalized characteristic forms and CR invariants}\label{renormalized-chara}
\subsection{Burns--Epstein's renormalized connection}\label{BE-conn}
Let $\Omega$ be a strictly pseudoconvex domain in an $(n+1)$-dimensional complex manifold $X$. We fix a Fefferman defining density $\br\in\wt\calE(1)$ and let $g$ be the Cheng--Yau metric, i.e., a hermitian metric on $\Omega$ which agrees with the K\"ahler metric $-i\pa\ol\pa\log\br$ near the boundary $M$. We denote by $\psi_i{}^j$ the connection 1-forms of the Chern connection of $g$; it agrees with the Levi-Civita connection near $M$. 

We assume that $M$ admits a pseudo-Einstein contact form $\th$, and take the associated Fefferman defining function $\r$. Burns--Epstein \cite{BE2} introduced a renormalized connection $\ol\nabla$ by 
\begin{equation}\label{theta-def}
\theta_i{}^j:=\psi_i{}^j+\frac{1}{\rho}(\rho_k\d_i{}^j+\rho_i \d_k{}^j)\theta^k,
\end{equation}
where $\pa\r=\r_j\th^j$, in an arbitrary $(1, 0)$-coframe $\{\th^j\}$. Note that $\ol\nabla$ depends on the choice of $\r$. Since $\r$ is determined by $\th$ near $M$, so is 
$\ol\nabla$. They showed that $\ol\nabla$ extends to a linear connection on $\ol\Omega$:

\begin{prop}[\cite{BE2}]\label{theta-in-z}
Let $(z^1, \dots, z^{n+1})$ be holomorphic coordinates around a point on the boundary $M$.
Then we have
\begin{equation}\label{theta-coordinate}
\theta_i{}^j=g^{j\ol l}\Bigl(-\frac{\rho_{ki\ol l}}{\rho}+\frac{\rho_{\ol l}\rho_{ki}}{\rho^2}\Bigr)dz^k,
\end{equation}
where $\rho_{ki\ol l}, \rho_{\ol l}, \rho_{ki}$ denote the coordinate derivatives. In particular, $\th_i{}^j$ is smooth up to the boundary.
\end{prop}
\begin{proof}
It suffices to show the identity
\[
\psi_i{}^j=g^{j\ol l}\Bigl(-\frac{\rho_{ki\ol l}}{\rho}+\frac{\rho_{\ol l}\rho_{ki}}{\rho^2}\Bigr)dz^k-\frac{1}{\rho}(\rho_k\d_i{}^j+\rho_i \d_k{}^j)dz^k.
\]
If we denote the right-hand side of this by $\psi'_i{}^j$, we have
\begin{align*}
\psi'_i{}^jg_{j\ol m}&=\Bigl(-\frac{\rho_{ki\ol m}}{\rho}+\frac{\rho_{\ol m}\rho_{ki}}{\rho^2}\Bigr)dz^k-\frac{1}{\rho}(g_{i\ol m}\rho_k+g_{k\ol m}\rho_i)dz^k \\
&=-\frac{\rho_{ki\ol m}}{\rho}dz^k+\frac{1}{\rho^2}(\rho_{i\ol m}\rho_k+\rho_{k\ol m}\rho_i+
\rho_{ki}\rho_{\ol m})dz^k 
 -\frac{2\rho_i\rho_{\ol m}\rho_k}{\rho^3}dz^k \\
&=\pa g_{i\ol m}.
\end{align*}
This implies $\psi'_i{}^j=\psi_i{}^j$. Since $g^{j\ol l}=O(\r)$ and  $g^{j\ol l}\r_{\ol l}=O(\r^2)$, the connection form $\th_i{}^j$ is smooth up to $M$.
\end{proof}

Next we consider the renormalized curvature. Let 
\[
\Psi_i{}^j=d\psi_i{}^j-\psi_i{}^k\wedge\psi_k{}^j, \quad 
\Theta_i{}^j=d\theta_i{}^j-\theta_i{}^k\wedge\theta_k{}^j
\]
be the curvature forms. Since $\psi_i{}^j$ is the Chern connection, in a holomorphic frame, $\psi_i{}^j$ and $\theta_i{}^j$ are $(1, 0)$-forms and $\Psi_i{}^j=\ol\pa\psi_i{}^j$. 
Then we have
\begin{align*}
\Psi_i{}^j=\ol\pa\Bigl(\theta_i{}^j-\frac{1}{\rho}(\rho_k \d_i{}^j+\rho_i \d_k{}^j)dz^k\Bigr) 
=W'_i{}^j-K'_i{}^j,
\end{align*}
where $W'_i{}^j:=(\Theta_i{}^j)^{(1, 1)}$ is the $(1, 1)$-component of $\Theta_i{}^j$ and 
\[
K'_i{}^j:=-\bigl((\log\r)_{k\ol l}\d_i{}^j+(\log\r)_{i\ol l}\d_k{}^j\bigr)dz^k\wedge dz^{\ol l}.
\]
Then $W'_i{}^j$ is smooth up to the boundary. On the other hand, Burns--Epstein \cite{BE2} defined the renormalized curvature of $g$ by 
\[
W_i{}^j:=\Psi_i{}^j+K_i{}^j, \ {\rm where}\ K_i{}^j:=(g_{k\ol l}\d_i{}^j+g_{i\ol l}\d_k{}^j)dz^k\wedge dz^{\ol l}.
\]
Since $g_{i\ol j}=-\pa_i\pa_{\ol j}\log\r$ near $M$, these coincide in a neighborhood of the boundary.

We will describe the relation between the renormalized connection $\ol\nabla$ and the Levi-Civita connection of the ambient metric.
Let $\wt g=-i\pa\ol\pa\br$ be the ambient metric on $\wt X$ near $\mathcal{N}$, and $\wt\omega_I{}^J$, $\wt\Omega_I{}^J$ be the connection and the curvature forms of $\wt g$ respectively. For the pseudo-Einstein contact form $\th$, we can take local coordinates $(z^1, \dots, z^{n+1})$ around a point on $M$ which are holomorphic in $\Omega$, in such a way that the associated pseudo-Einstein CR scale corresponds to $\th$.  We take a fiber coordinate $z^0$ of $\wt X$ by writing as $\zeta=(z^0)^{n+2}dz^1\wedge\cdots\wedge dz^{n+1}$ so that $\br=|z^0|^2\r$ as a function on $\wt X$. Then we have
\begin{prop}[{\cite[Lemma 3.2]{BE2}}]
In the holomorphic frame $\{dz^0/z^0, dz^1, \dots, dz^{n+1}\}$ around a boundary point, we have
\begin{gather}
\label{ambient-conn-coordinate}
\wt\omega_I{}^J=
\begin{pmatrix}
\displaystyle\frac{dz^0}{z^0} & u_i \\
dz^j & \theta_i{}^j+\displaystyle\frac{dz^0}{z^0}\d_i{}^j
\end{pmatrix}, \\
\label{ambient-curvature-coordinate}
\wt\Omega_I{}^J=
\begin{pmatrix}
0 & du_i+u_k\wedge\theta_i{}^k \\
0 & \Theta_i{}^j-u_i\wedge dz^j
\end{pmatrix} 
=\begin{pmatrix}
0 & du_i+u_k\wedge\theta_i{}^k \\
0 & W_i{}^j
\end{pmatrix},
\end{gather}
where
\[
u_i:=\frac{1}{\r}(\rho_{ik} dz^k-\rho_k\theta_i{}^k).
\]
\end{prop}
\begin{proof}
We denote the right-hand side of \eqref{ambient-conn-coordinate} by 
$\wt\omega'_I{}^J$. Then we have
\begin{align*}
\wt g_{K\ol J}\wt\omega'_I{}^K
&=-|z^0|^2\begin{pmatrix}
\rho & \rho_k \\
\rho_{\ol j} & \rho_{k\ol j}
\end{pmatrix}
\begin{pmatrix}
dz^0/z^0 & u_i \\
dz^k & \theta_i{}^k+(dz^0/z^0)\d_i{}^k
\end{pmatrix} \\
&=-|z^0|^2
\begin{pmatrix}
\rho (dz^0/z^0)+\pa\rho & \rho_i(dz^0/z^0)+\rho u_i+\rho_k\theta_i{}^k \\
\rho_{\ol j}(dz^0/z^0)+\pa\rho_{\ol j} & \rho_{i\ol j}(dz^0/z^0)+u_i\rho_{\ol j}+
\rho_{k\ol j}\theta_i{}^k
\end{pmatrix}.
\end{align*}
By \eqref{theta-coordinate}, the 1-forms in the matrix are computed as
\begin{align*}
\rho u_i+\rho_k\theta_i{}^k&=\rho_{ik}dz^k-\rho_k\theta_i{}^k+\rho_k\theta_i{}^k \\
&=\pa\rho_i, \\
u_i\rho_{\ol j}+\rho_{k\ol j}\theta_i{}^k
&=\r^{-1}(\rho_{ik}\rho_{\ol j}dz^k-\rho_k\rho_{\ol j}\theta_i{}^k)+\rho_{k\ol j}\theta_i{}^k \\
&=\r^{-1}\rho_{ik}\rho_{\ol j}dz^k-\rho g_{k\ol j}\theta_i{}^k \\
&=\pa \rho_{i\ol j}.
\end{align*}
Thus, we obtain $\wt g_{K\ol J}\wt\omega'_I{}^K=\pa\wt g_{I\ol J}$ and hence $\wt\omega'_I{}^J=\wt\omega_I{}^J$. 

Since $\ol\nabla$ is torsion-free near $M$, we have $dz^k\wedge \th_k{}^j=d(dz^k)=0$ and hence 
\begin{equation}\label{u-z}
u_k\wedge dz^k=0.
\end{equation}
By using these formulas and \eqref{ambient-conn-coordinate}, we obtain the first equality of \eqref{ambient-curvature-coordinate}. The second equality follows from the fact that 
$\wt\Omega_I{}^J$ are $(1, 1)$-forms and $(\Theta_i{}^j)^{(1, 1)}=W_i{}^j$ near $M$.
\end{proof}

Next, we will describe the Levi-Civita connection of $\wt g$ in terms of the Graham--Lee connection $\varphi_\a{}^\b$ for $\r$. Let $\{Z_\a, Z_\infty=\xi\}$ be a Graham--Lee frame near $M$ and $\{\th^\a, \th^\infty=\pa\r\}$ the dual coframe. In the coframe $\{\th^0=dz^0/z^0, \th^\a, \th^\infty\}$, the ambient metric is given by
\[
(\wt g_{A\ol B})=|z^0|^2
\begin{pmatrix}
-\r & 0 & -1 \\
0 & h_{\a\ol\b} & 0 \\
-1 & 0 & \kappa
\end{pmatrix}.
\]
We first recall from \cite{Mar} the relation between $\th_i{}^j$ and the Graham--Lee connection:
\begin{prop}[{\cite[Proposition 3.5]{Mar}}]
Near the boundary $M$, the renormalized connection forms $\th_i{}^j$ are given by 
\begin{equation}\label{theta-phi}
\begin{aligned}
\theta_\a{}^\b&=\varphi_\a{}^\b+i\kappa\vartheta\d_\a{}^\b,  & 
\theta_\infty{}^\b&=-\kappa\theta^\b+iA^\b{}_{\ol\g}\theta^{\ol\g}-\kappa^\b\ol\pa\rho, \\
\theta_\a{}^\infty&=h_{\a\ol\g}\theta^{\ol\g}+\mu\rho\kappa_\a\pa\rho+i\mu\rho A_{\a\g}\theta^\g, & 
\theta_\infty{}^\infty&=\kappa\ol\pa\rho-\mu\kappa^2\rho\pa\rho
+\mu\rho\pa\kappa,
\end{aligned}
\end{equation}
where $\mu:=(1+\kappa\r)^{-1}$.
\end{prop}
When $\r$ is a Fefferman defining function, this proposition and Proposition \ref{exp-GL} allow us to relate geometric quantities of $\ol\nabla$ to the Tanaka--Webster curvature quantities on $M$.

In a Graham--Lee frame, we also have
\begin{align}
\label{d-h}
d h_{\a\ol\b}&=\theta_\a{}^\g h_{\g\ol\b}+h_{\a\ol\g} \theta_{\ol\b}{}^{\ol\g}, \\
\label{d-kappa-rho}
\mu d(\kappa\rho)&=\theta_\infty{}^\infty+\theta_{\ol\infty}{}^{\ol\infty}, \\
\label{d-theta}
d\theta^i&=\theta^j\wedge\theta_j{}^i.
\end{align}
The equation \eqref{d-h} is shown in the proof of {\cite[Proposition 3.5]{Mar}}, and the equation \eqref{d-theta} follows from the fact that $\ol\nabla$ is torsion-free near $M$.
The equation \eqref{d-kappa-rho} is derived by
\begin{align*}
\theta_\infty{}^\infty+\theta_{\ol\infty}{}^{\ol\infty}=\psi_\infty{}^\infty+\psi_{\ol\infty}{}^{\ol\infty}
+2\r^{-1}d\rho=g_{\infty\ol\infty}^{-1}\,d g_{\infty\ol\infty}+2\r^{-1}d\rho.
\end{align*} 
Using these equations, we obtain the following 
\begin{prop} \label{ambient-connection}
In the coframe $\{\th^0=dz^0/z^0, \th^\a, \th^\infty\}$, the connection 1-forms $\wt\omega_A{}^B$ of the ambient metric $\wt g$ are given by
\begin{align*}
\wt\omega_0{}^0&=\frac{dz^0}{z^0}, &  \wt\omega_0{}^\a&=\theta^\a, 
 & \wt\omega_\b{}^\a&=\theta_\b{}^\a+\frac{dz^0}{z^0}\d_\b{}^\a, &   \wt\omega_\infty{}^\a&=\theta_\infty{}^\a, \\
\wt\omega_0{}^\infty&=\pa\rho, & \wt\omega_\b{}^\infty&=\theta_\b{}^\infty, & 
\wt\omega_\infty{}^\infty&=\theta_\infty{}^\infty+\frac{dz^0}{z^0}, 
\end{align*}
\begin{align*}
\wt\omega_\a{}^0&=-\r^{-1}(\theta_\a{}^\infty-h_{\a\ol\g}\theta^{\ol\g})
=-\mu\kappa_\a\pa\rho-i\mu A_{\a\g}\theta^\g, \\
\wt\omega_\infty{}^0&=-\r^{-1}(\theta_\infty{}^\infty-\kappa\ol\pa\rho)=
\mu\kappa^2\pa\rho-\mu\pa\kappa.
\end{align*}
\end{prop}
\begin{proof}
One can check that $\wt\omega_A{}^B$ given above satisfies the structure equations
\[
d\theta^A=\theta^B\wedge \wt\omega_B{}^A, \quad d\,\wt g_{A\ol B}=\wt\omega_A{}^C\wt g_{C\ol B}+\wt g_{A\ol C}\wt\omega_{\ol B}{}^{\ol C}
\]
by using \eqref{d-h}, \eqref{d-kappa-rho}, and \eqref{d-theta}.
\end{proof}
In particular, at $\rho=0$ we have
\begin{equation}\label{omega-rho-0}
\begin{gathered}
\wt\omega_0{}^0=\frac{dz^0}{z^0}, \quad \wt\omega_0{}^\a=\theta^\a, \quad  
\ \wt\omega_\b{}^\a=\varphi_\b{}^\a+\Bigl(\frac{dz^0}{z^0}+i\kappa\vartheta\Bigr)\d_\b{}^\a, \\ 
\wt\omega_\infty{}^\a=-\kappa\theta^\a+iA^\a{}_{\ol\g}\theta^{\ol\g}-\kappa^\a\ol\pa\rho, \quad 
\wt\omega_0{}^\infty=\pa\rho, \quad  \wt\omega_\b{}^\infty=h_{\b\ol\g}\theta^{\ol\g}, \\
\wt\omega_\infty{}^\infty=\kappa\ol\pa\rho+\frac{dz^0}{z^0}, \quad
\wt\omega_\a{}^0=-\kappa_\a\pa\rho-iA_{\a\g}\theta^\g, \quad 
\wt\omega_\infty{}^0=\kappa^2\pa\rho-\pa\kappa.
\end{gathered}
\end{equation}

%%%%%%%%%%%%%%%%%%%%%%%%%%%%%%%%%%%%%%%%%%%%%%%%%%%%%%%%%%%%%%%%%%%%%%%%5
\subsection{Renormalized characteristic forms}\label{ren-chara}
Let $\Phi$ be an Ad-invariant homogeneous polynomial of degree $m$. It is uniquely expressed as a linear combination of terms of the form $T_{m_1}T_{m_2}\cdots T_{m_k}\ (m_1+\cdots+m_k=m)$, where 
\begin{equation}\label{def-T}
T_p(A):={\rm tr}(iA)^p.
\end{equation}
 We omit the factor $2\pi$ in the usual definition of the characteristic forms. By using this expression, we regard $\Phi$ as an Ad-invariant polynomial on $\mathfrak{gl}(r, \mathbb{C})$ for any $r$. Let $\Theta_i{}^j$ be the curvature form of the renormalized connection associated with a Fefferman defining function $\r$. Then, $\Phi(\Theta)$ defines a smooth closed $2m$-form on $\ol\Omega$, which we call the {\it renormalized characteristic form}. In \cite{BE2}, they define the renormalized characteristic form as $\Phi(W)$, using $W_i{}^j$ in place of $\Theta_i{}^j$. Near the boundary $M$, these forms coincide:

\begin{prop}\label{Theta-W}
{\rm (i)}We have $\Phi(\Theta)=\Phi(W)=\Phi(W')$ near $M$. \\
{\rm (ii)}When $m=n+1$, we have $\Phi(\Theta)=\Phi(W')$ on $\Omega$.
\end{prop}
\begin{proof}
(i) First we recall that $W_i{}^j=W'_i{}^j$ near $M$. We take local coordinates $(z^1, z^2, \dots, z^{n+1})$ around a point on the boundary. By \eqref{ambient-curvature-coordinate}, we can decompose $\Theta_i{}^j$ as
\[
\Theta_i{}^j=W_i{}^j+U_i{}^j
\]
with
\[
U_i{}^j:=u_i\wedge dz^j, \quad u_i=\r^{-1}(\rho_{ik}dz^k-\rho_k\theta_i{}^k).
\]
By \eqref{u-z}, we have
\begin{equation}\label{tr-U}
U_i{}^i=0, \quad U_i{}^k \wedge U_k{}^j=0.
\end{equation}
Moreover, since $W$ can be written as
\[
W_i{}^j=W_i{}^j{}_{k\ol l}dz^k\wedge dz^{\ol l}, \quad W_i{}^j{}_{k\ol l}=W_k{}^j{}_{i\ol l}
\]
near $M$, we also have $dz^i\wedge W_i{}^j=0$ and hence
\begin{equation}\label{U-W}
U_i{}^k\wedge  W_k{}^j=0.
\end{equation}
It follows from \eqref{tr-U} and \eqref{U-W} that
\[
T_p(\Theta)=T_p(W)
\]
for any $p$. Thus we obtain $\Phi(\Theta)=\Phi(W)$ near $M$. 

(ii) Since $\th_i{}^j$ is a $(1, 0)$-form in a holomorphic frame, the curvature $\Theta_i{}^j$ does not contain the $(0, 2)$-component. This implies that when $\Phi$ is of degree $n+1$, we have $\Phi(\Theta)=\Phi(\Theta^{(1, 1)})=\Phi(W')$ on $\Omega$.
\end{proof}
By this proof, we have $\Phi(\Theta)=\Phi(W)=\Phi(W')$ globally on $\Omega$ if $g=-i\pa\ol\pa\log\rho$ globally. Thus, when $X=\C^{n+1}$ and $g$ is the exact Cheng--Yau metric, 
the renormalized characteristic forms depend only on $g$ and are biholomorphically invariant. 

By \eqref{ambient-curvature-coordinate}, we also have the identity
\begin{equation}\label{W-Omega}
\Phi(W)=\Phi(\wt\Omega),
\end{equation}
where $\wt\Omega$ is the curvature form of the ambient metric $\wt g$ and the left-hand side is implicitly pulled back via the projection $\wt X\to X$. Hence the renormalized characteristic forms are nothing other than the characteristic forms of the ambient metric. 

Though $\Phi(\Theta)$ depends on the choice of $\r$ off the boundary in general cases, it has an advantage that it is a closed form on whole $\Omega$ since $\Theta$ is the curvature form of a linear connection on $T^{1, 0}X$. We note that $T^{1, 0}X|_{M}$ is isomorphic to the direct sum of $T^{1, 0}M$ and a trivial complex line bundle. Thus, the cohomology class $[\Phi(\Theta)]$ on a collar neighborhood of $M$ is determined by $\Phi(T^{1, 0}M)$. When the degree of $\Phi$ is sufficiently large, this cohomology class always vanishes:
\begin{thm}[{\cite[Theorem 1.1]{T}}]\label{chern-cr}
Let $M$ be a closed strictly pseudoconvex CR manifold of dimension $2n+1\ge 5$. Let $\Phi$ be an invariant polynomial of degree $m$ with $2m\ge n+2$. Then $\Phi(T^{1, 0}M)=0$ in $H^{2m}(M; \mathbb{R})$.
\end{thm}
\begin{cor}\label{cor-chern-cr}
Let $n\ge 2$. If $\Phi$ is an invariant polynomial of degree $n$, then $\Phi(\Theta)$ is exact near $M$.
\end{cor}

When $g$ satisfies the Einstein equation $\Ric(g)+(n+2)g=0$ exactly, the renormalized characteristic forms can be rewritten in terms of those of $g$:
Let $\omega_g=ig_{k\ol l}\theta^k\wedge\theta^{\ol l}$ be the K\"ahler form of $g$. Then we have 
\begin{equation}\label{K-Psi}
K_i{}^l\Psi_l{}^j=\Psi_i{}^l K_l{}^j=-i\omega_g \Psi_i{}^j, \quad K_i{}^l K_l{}^j=-i\omega_g K_i{}^j,
\quad K_l{}^l=-(n+2)i\omega_g.
\end{equation}
Moreover, the Einstein equation implies $c_1=-(n+2)\omega_g$. By expanding 
$T_m(W)={\rm tr}\,\bigl(i(\Psi+K)\bigr)^m$, we obtain the following proposition:

\begin{prop}[{\cite[Theorem 2.1]{BE2}}]
If $g$ satisfies ${\rm Ric}_{i\ol j}+(n+2)g_{i\ol j}=0$, then we have $T_1(W)=0$ and $T_m(W)=\wt T_m(\Psi)$ for $m\ge2$, where 
\[
\wt T_m=\sum_{l=0}^{m-2}(-1)^l\binom{m}{l}\frac{1}{(n+2)^l}c_1^{\,l} T_{m-l}+(-1)^{m-1}\frac{m-1}{(n+2)^{m-1}}c_1^m.
\]
\end{prop}
It follows from this proposition that for any invariant polynomial $\Phi$ of degree $m$, we have 
\[
\Phi(W)=\wt\Phi(\Psi)
\]
with an invariant polynomial $\wt\Phi$ of the form
\[
\wt\Phi=\Phi+\sum_{l=1}^{m}c_1^{\, l}\phi_l. 
\]
For example, we have
\[
\wt c_3=\frac{1}{3}\wt T_3=c_3-\frac{n}{n+2}c_1c_2+\frac{n(n+1)}{3(n+2)^2}c_1^3.
\]

When $\r$ is a Fefferman defining function and $g_{i\ol j}=-\pa_i\pa_{\ol j}\log\r$ holds only near the boundary, it is more convenient to use $\omega:=-i\pa\ol\pa\log\r$ and $W'=\Theta^{(1, 1)}=\Psi+K'$ instead of the K\"ahler form $\omega_g$ and the renormalized curvature $W$. By using similar equations as \eqref{K-Psi}, we can expand as
\begin{equation}\label{T-W-prime}
T_m(W')=\sum_{l=0}^{m-1}\binom{m}{l}\omega^l\wedge T_{m-l}(\Psi)+(n+2)\omega^m.
\end{equation}
Interchanging $\Psi$ and $W'$, we also have
\begin{equation}\label{T-Psi}
T_m(\Psi)=\sum_{l=0}^{m-1}(-1)^l\binom{m}{l}\omega^l\wedge T_{m-l}(W')+(-1)^m(n+2)\omega^m.
\end{equation}
By \eqref{T-W-prime}, \eqref{T-Psi}, and Proposition \ref{Theta-W} (ii), we obtain 
the expansion \eqref{chern-exp} for $c_{n+1}(\Theta)$. Also, the \eqref{T-Psi} implies that $\omega^{n+1-m}\wedge\Phi(\Psi)$ for an invariant polynomial $\Phi$ of degree  $m$ can be expressed in terms of $\omega$ and renormalized characteristic forms.

%%%%%%%%%%%%%%%%%%%%%%%%%%%%%%%%%%%%%%%%%%%%%%%%%%%%%%%%%%%%%%%%%%%%%%%%%

\subsection{Constructions of CR invariants}
We will now prove Theorem \ref{integral-inv}. We first prove that the  left-hand side of \eqref{finite-part} is independent of the choice of Fefferman defining function $\rho$. To this end, we rewrite it to the logarithmic part of another integral:

\begin{prop}\label{fp-lp}
For any Fefferman defining function $\rho$, we have 
\begin{equation}\label{fp-lp-formula}
{\rm fp} \int_{\rho>\e} \omega^{n+1-m}\wedge\Phi(\Theta)
=-{\rm lp}\int_{\rho>\e} i\pa\log\rho\wedge\ol\pa\log\rho\wedge\omega^{n-m}\wedge\Phi(W).
\end{equation}
\end{prop}
\begin{proof}
Since the logarithmic term is determined by the boundary behavior of the integrand, we may replace $\Phi(W)$ in the right-hand side of \eqref{fp-lp-formula} by $\Phi(\Theta)$. 
Using $d\Phi(\Theta)=0$, we have
\begin{align*}
&\quad \int_{\rho>\e}i \pa\log\rho\wedge\ol\pa\log\rho\wedge\omega^{n-m}\wedge\Phi( \Theta) \\
&=\int_{\rho>\e}i d\log\rho\wedge\ol\pa\log\rho\wedge\omega^{n-m}\wedge\Phi(\Theta) \\
&=\log\e\int_{\rho=\e} i\ol\pa\log\rho\wedge\omega^{n-m}\wedge\Phi(\Theta) 
 -\int_{\rho>\e} i\log\rho\ \pa\ol\pa\log\rho\wedge\omega^{n-m}\wedge\Phi(\Theta) \\
&=-\log\e\int_{\rho>\e} \omega^{n+1-m}\wedge\Phi(\Theta)+\int_{\rho>\e} \log\rho\ \omega^{n+1-m}\wedge\Phi(\Theta).
\end{align*}
The coefficient of $\log\e$ in the expansion of the first term is given by
\[
-{\rm fp} \int_{\rho>\e} \omega^{n+1-m}\wedge\Phi(\Theta).
\]
By using the flow generated by $N$, we identify a neighborhood of $M$ in $\ol\Omega$ with $M\times[0, \e_0)_\rho$. Then, we can expand as
\[
\log\rho\ \omega^{n+1-m}\wedge\Phi(\Theta)=\Bigl(\,\sum_{k=0}^{p}
\rho^{-k}(\log\rho)a_k(x)+O(1)\Bigr)d\rho\wedge\theta\wedge(d\theta)^n,
\]
where $\theta=(i/2)(\pa\r-\ol\pa\r)|_{TM}$ and $a_k(x)$ is a smooth function on $M$. Noting that 
\[
\int \rho^{-k}\log\rho\ d\rho=
\begin{cases}
\displaystyle\frac{1}{2}(\log\rho)^2 & (k=1) \\
\displaystyle\frac{1}{1-k}\rho^{1-k}\log\rho-\frac{1}{(1-k)^2}\rho^{1-k} & (k\neq1),
\end{cases}
\]
we find that 
\[
{\rm lp}\int_{\rho>\e} \log\rho\ \omega^{n+1-m}\wedge\Phi(\Theta)=0.
\]
Thus we obtain \eqref{fp-lp-formula}.
\end{proof}

The advantage of rewriting the finite part of an integration to the log part of another integration is that the latter does not depend on the choice of defining function to shrink the domain, as the following lemma shows:  
\begin{lem}[cf. {\cite[Theorem 3.1]{Gr3}}]\label{independence}
Let $\rho$ be a Fefferman defining function and $r$ an arbitrary defining function of $\Omega$. Then
\[
{\rm lp}\int_{r>\e} i\pa\log\rho\wedge\ol\pa\log\rho\wedge\omega^{n-m}\wedge\Phi(W)
\]
is independent of the choice of $r$.
\end{lem}
\begin{proof}
Using a flow, we identify a neighborhood of $M$ in $\ol\Omega$ with $M\times[0, \e_0)_r$ and write the integrand as
\[
 i\pa\log\rho\wedge\ol\pa\log\rho\wedge\omega^{n-m}\wedge\Phi(W)
=\psi(x, r)dr\wedge\theta\wedge(d\theta)^n
\]
where $\psi$ admits the Laurent expansion in $r$. Let $\wh r$ be another defining function. Then $(x, \wh r)$ give local coordinates and we can write as 
\[
r=b(x, \wh r)\wh r
\]
with $b$ smooth up to the boundary. We have
\begin{align*}
&\quad \int_{\wh r\,>\e}\psi(x, r)dr\wedge\theta\wedge(d\theta)^n-
\int_{r>\e}\psi(x, r)dr\wedge\theta\wedge(d\theta)^n \\
&= \int_{b(x, \e)\e}^\e\int_M \psi(x, r)dr\wedge\theta\wedge(d\theta)^n
\end{align*}
and this has no logarithmic term. Thus we obtain the lemma.
\end{proof}

\begin{prop}\label{invariance}
Let $\rho, \wh \rho$ be Fefferman defining functions. Then we have
\begin{equation}\label{invariance-formula}
\begin{aligned}
& {\rm lp}\int_{\wh\rho\,>\e} i\pa\log\wh\rho\wedge\ol\pa\log\wh\rho\wedge\wh\omega^{n-m}\wedge\Phi(W)  \\
&\quad\quad={\rm lp}\int_{\rho>\e} i\pa\log\rho\wedge\ol\pa\log\rho\wedge\omega^{n-m}\wedge\Phi(W).
\end{aligned}
\end{equation}
\end{prop}
\begin{proof}
By Lemma \ref{independence}, we may replace the domain of integration in the left-hand side by $\{\rho>\e\}$. By Proposition \ref{pluriharmonic-ext}, we can write $\wh \rho=e^{\U}\rho$ with $\U\in C^\infty(\ol\Omega)$ such that $\pa\ol\pa\U=0$ near $M$. Thus, we have
\begin{align*}
&\quad \int_{\rho\,>\e} i\pa\log\wh\rho\wedge\ol\pa\log\wh\rho\wedge\wh\omega^{n-m}\wedge\Phi(W) \\
&=\int_{\rho>\e} i\pa\log\rho\wedge\ol\pa\log\rho\wedge\omega^{n-m}\wedge\Phi(W)  
+\int_{\rho>\e} i\pa\U\wedge\ol\pa\U\wedge\omega^{n-m}\wedge\Phi(W) \\
&\quad +2{\rm Re}\int_{\rho>\e} i\pa\log\rho\wedge\ol\pa\U\wedge\omega^{n-m}\wedge\Phi( W)+\int_{\rho>\e} ({\rm cpt\ supp}),
\end{align*}
where $({\rm cpt\ supp})$ denotes a compactly supported form on $\Omega$.
Since $\Phi(W)$ is closed near $M$, the second term in the right-hand side is equal to 
\[
\int_{\rho=\e}i \U\ol\pa\U\wedge\omega^{n-m}\wedge\Phi(W) 
+\int_{\rho>\e} ({\rm cpt\ supp}),
\] 
which gives no $\log\e$ term. On the other hand, noting that $\Phi(W)$ is an $(m, m)$-form on $\Omega$, we compute as
\begin{align*}
&\quad \int_{\rho>\e} i\pa\log\rho\wedge\ol\pa\U\wedge\omega^{n-m}\wedge\Phi(W)
\\
&=\int_{\rho>\e} id\log\rho\wedge\ol\pa\U\wedge\omega^{n-m}\wedge\Phi(W) \\
&=i\log\e\int_{\rho=\e} \ol\pa\U\wedge\omega^{n-m}\wedge\Phi(W)
+\int_{\rho>\e} ({\rm cpt\ supp}).
\end{align*}
Here, the integrand of the first term in the last expression is exact near $M$: When $m<n$, this follows from the exactness of $\omega$, and when $m=n$ Proposition \ref{Theta-W} (i) and Corollary \ref{cor-chern-cr} imply the exactness of $\Phi(W)$. Thus, this term does not have $\log\e$ term either. Hence we obtain \eqref{invariance-formula}.
\end{proof}

\begin{rem}\label{rem-kahler}
When $n\ge 2$, we can embed $M$ as the boundary of a strictly pseudoconvex domain $\Omega$ in a K\"ahler manifold $X$ ({\cite[Theorem 8.1]{Lem}}), and a CR pluriharmonic function on $M$ is extended to a global pluriharmonic function on $\Omega$ ({\cite[Theorem 7.1]{H}}). Then, the last part of the proof can also be given as follows:
\begin{align*}
\int_{\rho=\e} \ol\pa\U\wedge\omega^{n-m}\wedge\Phi(W)&=\int_{\rho=\e} \ol\pa\U\wedge\omega^{n-m}\wedge\Phi(\Theta) \\
&=\int_{\rho>\e} d\bigr(\ol\pa\U\wedge\omega^{n-m}\wedge\Phi(\Theta)\bigl) \\
&=0.
\end{align*} 
\end{rem}
\bigskip

It follows from Propositions \ref{fp-lp} and \ref{invariance} that  the  left-hand side of \eqref{finite-part} is independent of the choice of Fefferman defining function $\rho$.

Next we will show that the finite part is expressed in terms of pseudo-hermitian invariants as in \eqref{finite-part}. Since $\omega=d(\vartheta/\rho)$ and $d\Phi(\Theta)=0$, we have
\begin{equation}\label{stokes}
\int_{\rho>\e} \omega^{n+1-m}\wedge\Phi(\Theta)=
\int_{\rho=\e} \e^{-n-1+m}\vartheta\wedge(d\vartheta)^{n-m}\wedge\Phi(\Theta).
\end{equation}
If we identify a neighborhood of $M$ in $\ol\Omega$ with $M\times[0, \e_0)_\rho$ by the flow of $N$, we can write as
\[
\vartheta\wedge(d\vartheta)^{n-m}\wedge\Phi(\Theta)=F(x, \rho)\theta\wedge(d\theta)^n+d\rho\wedge\eta
\]
with a smooth function $F$ and a $2n$-form $\eta$. Then, by \eqref{stokes} we have
\[
{\rm fp}\int_{\rho>\e} \omega^{n+1-m}\wedge\Phi(\Theta)=\frac{1}{(n+1-m)!}\int_M
(\pa^{n+1-m}_\rho F)(x, 0)\theta\wedge(d\theta)^n.
\]
On the other hand, using $\mathcal{L}_N d\rho=0$, we have
\[
\mathcal{L}^{n+1-m}_N \bigl(\vartheta\wedge(d\vartheta)^{n-m}\wedge\Phi(\Theta)\bigr)
=(\pa^{n+1-m}_\rho F)(x, \rho)\theta\wedge(d\theta)^n+d\rho\wedge\mathcal{L}^{n+1-m}_N\eta.
\]
Hence we obtain 
\begin{equation}\label{int-lie}
{\rm fp}\int_{\rho>\e} \omega^{n+1-m}\wedge\Phi(\Theta)
=\frac{1}{(n+1-m)!}\int_M\mathcal{L}^{n+1-m}_N \bigl(\vartheta\wedge(d\vartheta)^{n-m}\wedge\Phi(\Theta)\bigr).
\end{equation}
Since $\mathcal{L}_N d\r=0$, we can ignore terms which contain $d\r$ in the computation of the integrand of the right-hand side. By $\mathcal{L}_N \vartheta =-\kappa\vartheta$ and $\mathcal{L}_N d\vartheta\equiv -\kappa 
d\vartheta$ mod $\vartheta$, we have
\[
\mathcal{L}_N^p \bigl(\vartheta\wedge (d\vartheta)^{n-m}\bigr)=\psi(\kappa, N\kappa, \dots, N^{p-1}\kappa)\, \vartheta\wedge (d\vartheta)^{n-m}
\]
with some polynomial $\psi$. On the other hand, by \eqref{theta-phi}, we may write as  
\[
\Phi(\Theta)\equiv G(h_{\a\ol\b}, R_{\a\ol\b\g\ol\mu}, A_{\a\b}, \kappa, \r, \nabla) (d\vartheta)^m
\quad {\rm mod}\ d\r, \vartheta,
\]
in which the covariant derivatives are all tangential. We note that when we apply $\mathcal{L}_N$ to the function $G$, we may replace it by $\nabla_N$. The commutators of $\nabla_N$ and tangential differentiations can be computed with the torsion and curvature tensors of the Graham--Lee connection, expressed by $h_{\a\ol\b}, R_{\a\ol\b\g\ol\mu}, A_{\a\b}, \kappa$ and their tangential derivatives.  Thus, we can express the integrand in the right-hand side of \eqref{int-lie} by
\begin{equation}\label{terms}
h_{\a\ol\b}, \quad \nabla_N^{p}R_{\a\ol\b\g\ol\mu}, \quad \nabla_N^{p} A_{\a\b}, \quad N^{p-1} \kappa \quad (p\le n+1)
\end{equation}
and their tangential derivatives. By Proposition \ref{exp-GL}, these can be written in terms of Tanaka--Webster curvature quantities. Thus we complete the proof of Theorem \ref{integral-inv}.
\bigskip

We remark that when $m=1$, i.e., $\Phi$ is a multiple of $T_1=c_1$, the corresponding invariant vanishes; in fact, from 
\[
T_1(W)=i(\Ric_{j\ol k}+(n+2)g_{j\ol k})\,\th^j\wedge\th^{\ol k}=\pa\ol\pa\, O(\r^{n+2})
\]
and 
\[
\pa\log\r\wedge\ol\pa\log\r\wedge\omega^{n-1}=\r^{-n-1}\pa\r\wedge\ol\pa\r\wedge (d\vartheta)^{n-1},
\]
we see that the right-hand side of \eqref{fp-lp-formula} vanishes. By a similar estimate, we can ignore the terms which contain $T_1$ when we express $\Phi$ as a linear combination of polynomials of the form $T_{m_1}\cdots T_{m_k}$. If $\Phi\neq 0$ modulo 
$T_1$, then our invariants turn out to be non-trivial; see Remark \ref{general-case}.

%%%%%%%%%%%%%%%%%%%%%%%%%%%%%%%%%%%%%%%%%%%%%%%%%%%%%%%%%%%%%%%%%%%%%%%%%
\subsection{The case of the exact Cheng--Yau metric}\label{exact}
We consider the case where $X=\C^{n+1}$ and $g$ is the exact K\"ahler--Einstein metric 
$\mathring g_{i\ol j}=-\pa_i\pa_{\ol j}\log u$ defined via the Monge--Amp\`ere solution of Cheng--Yau  \cite{ChY}.
In this case, $u$ has a logarithmic singularity of the form
\begin{equation}\label{u-expansion}
u\sim \rho\sum_{j=0}^\infty \eta_j(\rho^{n+2}\log\rho)^j, \quad \eta_j\in C^\infty(\ol\Omega),
\end{equation}
where $\r\in C^{\infty}(\ol\Omega)$ is an arbitrary smooth defining function (\cite{LM}).
 We take a Fefferman defining function as $\r$, in which case we have $\eta_0=1+O(\rho^{n+2})$. 
We will show that we may use this exact Cheng--Yau metric to obtain the same global CR invariants in Theorem \ref{integral-inv}.

For a function $f\in C^\infty(\Omega)$, we write
\[
f=O(\rho^m\log\rho)
\]
if $f$ admits an expansion of the form
\[
f\sim a_0 \rho^m +a_1\rho^m\log\rho +\sum_{j=m+1}^\infty\sum_{l=0}^{k_j} b_{j, l}\rho^{\,j}(\log\rho)^l
\]
with $a_0, a_1, b_{j, l}\in C^\infty(\ol\Omega)$. Then, by \eqref{u-expansion}, we have
\begin{equation}\label{u-rho}
u=\rho(1+O(\rho^{n+2}\log\rho))=\rho+O(\rho^{n+3}\log\rho).
\end{equation}
Let $\mathring\psi_i{}^j, \mathring\Omega_i{}^j$ be the connection and the curvature forms of the Levi-Civita connection form of $\mathring g$. As before we define the renormalized connection by
\[
\mathring\theta_i{}^j:=\mathring\psi_i{}^j+\frac{1}{u}(u_k\d_i{}^j+u_i\d_k{}^j)\theta^k.
\]
We also define 
\[
\mathring W_i{}^j:=\mathring\Omega_i{}^j+(\mathring g_{k\ol l}\d_i{}^j+\mathring g_{i\ol l}\d_k{}^j)\theta^k\wedge\theta^{\ol l},
\]
which coincides with the $(1, 1)$-part of the curvature form $\mathring\Theta_i{}^j$ of $\mathring\theta_i{}^j$. We will examine the boundary regularity of these forms by using the formula from Proposition \ref{theta-in-z}
\[
\mathring\theta_i{}^j=\mathring g^{j\ol l}\Bigl(-\frac{u_{ki\ol l}}{u}+\frac{u_{\ol l}u_{ki}}{u^2}\Bigr)dz^k
\]
in a local coordinate system $(z^1, \dots, z^{n+1})$.
We set $g_{i\ol j}:=-\pa_i \pa_{\ol j}\log\rho$ so that 
\[
\mathring g=g+\pa\ol\pa\, O(\rho^{n+2}\log\rho).
\]
In a Graham--Lee coframe $\{\theta^\a, \theta^\infty=\pa\rho\}$, we have
\[
g=\r^{-1}h_{\a\ol\b}\theta^\a\wedge\theta^{\ol \b}+\r^{-2}(1+\kappa\rho)\pa\rho\wedge\ol\pa\rho
\]
and thus
\begin{gather*}
\mathring g_{\a\ol\b}=\r^{-1}h_{\a\ol\b}+O(\rho^{n+2}\log\rho),  \quad
\mathring g_{\a\ol\infty}=O(\rho^{n+1}\log\rho), \\
\mathring g_{\infty\ol\infty}=\r^{-2}(1+\kappa\rho)+O(\rho^{n}\log\rho).
\end{gather*}
Hence we have
\[
\mathring g=(I+A)g
\]
with a matrix $A=O(\rho^{n+1}\log\rho)$. Note that this also holds in the coframe 
$\{dz^j\}$. Tanking the inverse we obtain
\[
\mathring g^{-1}=g^{-1}(I+A'), \quad A'=O(\rho^{n+1}\log\rho),
\]
which gives
\begin{align*}
\mathring g^{j\ol l}&=g^{j\ol l}+O(\rho^{n+2}\log\rho) 
=O(\rho)+O(\rho^{n+2}\log\rho), \\
\mathring g^{j\ol l}\rho_{\ol l}&=g^{j\ol l}\rho_{\ol l}+O(\rho^{n+3}\log\rho) 
=O(\rho^2)+O(\rho^{n+3}\log\rho).
\end{align*}
It follows that 
\begin{align*}
\mathring g^{j\ol l}u_{ki\ol l}
&=\mathring g^{j\ol l}\bigl(\rho_{ki}+O(\rho^{n+1}\log\rho)\bigr)_{\ol l} 
=g^{j\ol l}\rho_{ki\ol l}+O(\rho^{n+2}\log\rho), \\
\mathring g^{j\ol l}u_{\ol l}&=\mathring g^{j\ol l}(\rho_{\ol l}+O(\rho^{n+2}\log\rho))
=g^{j\ol l}\rho_{\ol l}+O(\rho^{n+3}\log\rho).
\end{align*}
Therefore we get
\begin{align*}
\mathring \theta_i{}^j&=\mathring g^{j\ol l}\Bigl(-\frac{u_{ki\ol l}}{u}+\frac{u_{\ol l}u_{ki}}{u^2}\Bigr)dz^k \\
&=\frac{-g^{j\ol l}\rho_{ki\ol l}+O(\rho^{n+2}\log\rho)}{\rho(1+O(\rho^{n+2}\log\rho))}
+\frac{(g^{j\ol l}\rho_{\ol l}+O(\rho^{n+3}\log\rho))(\rho_{ki}+O(\rho^{n+1}\log\rho))}{\rho^2(1+O(\rho^{n+2}\log\rho))} \\
&=\theta_i{}^j+O(\rho^{n+1}\log\rho)
\end{align*}
and 
\begin{align*}
\mathring\Theta_i{}^j=\Theta_i{}^j+O(\rho^{n}\log\rho).
\end{align*}

Now we will prove that in Theorem \ref{integral-inv} we can replace the metric $g$ by $\mathring g$:

\begin{prop}
Let $\Phi$ be an invariant polynomial of degree $m\ (2\le m\le n)$. Then we have
\[
{\rm fp}\int_{u>\e}\omega_{\mathring g}^{n+1-m}\wedge \Phi(\mathring\Theta)
={\rm fp}\int_{\rho>\e}\omega^{n+1-m}\wedge \Phi(\Theta).
\]
\end{prop}
\begin{proof}
First we prove that we may replace $u$ in the left-hand side by $\rho$. We identify a neighborhood of $M$ in $\ol\Omega$ with $M\times [0, \e_0)_\rho$ and write
\[
\omega_{\mathring g}^{n+1-m}\wedge \Phi(\mathring\Theta)=f(x, \rho)d\rho\wedge\theta\wedge(d\theta)^n.
\]
Here the function $f$ is $O(\rho^{-(n+1-m)-1})$ and may involve some logarithmic terms. Let $\phi_\e(x)$ be a function on $M$ defined by the equation $u(x, \phi_\e(x))=\e$ for small $\e>0$. Then, $\phi_\e(x)=\e(1+O(\e^{n+2}\log\e))$, and we have
\begin{align*}
&\int_{u>\e}\omega_{\mathring g}^{n+1-m}\wedge \Phi(\mathring\Theta)-\int_{\rho>\e}\omega_{\mathring g}^{n+1-m}\wedge \Phi(\mathring\Theta) \\
&\quad =
\int_M \int_{\phi_\e(x)}^\e f(x, \rho)d\rho\wedge\theta\wedge(d\theta)^n.
\end{align*}
Since
\begin{align*}
\int_{\phi_\e(x)}^\e f(x, \rho)d\rho&=O([\rho^{-(n+1-m)}]_{\e(1+O(\e^{n+2}\log\e))}^\e) \\
&=\e^{-(n+1-m)}\cdot O(\e^{n+2}\log\e) \\
&=O(\e^{m+1}\log\e),
\end{align*}
we have
\[
{\rm fp}\int_{u>\e}\omega_{\mathring g}^{n+1-m}\wedge \Phi(\mathring\Theta)={\rm fp}\int_{\rho>\e}\omega_{\mathring g}^{n+1-m}\wedge \Phi(\mathring\Theta)
\]
Next, we compare the integrands. From $\mathring\Theta_i{}^j=\Theta_i{}^j+O(\rho^{n}\log\rho)$, we have
\[
\Phi(\mathring\Theta)=\Phi(\Theta)+O(\rho^{n}\log\rho),
\]
and the term $O(\rho^{n}\log\rho)$ is closed. From $\omega_{\mathring g}=\omega+\pa\ol\pa\, O(\rho^{n+2}\log\rho),
$ we have
\[
\omega_{\mathring g}^{n+1-m}=\omega^{n+1-m}+d\, O(\rho^{m+1}\log\rho).
\]
Consequently, we obtain
\[
\omega_{\mathring g}^{n+1-m}\wedge \Phi(\mathring\Theta)=
\omega^{n+1-m}\wedge \Phi(\Theta)+d\,O(\rho^{m-1}\log\rho)
\]
and complete the proof.
\end{proof}

%%%%%%%%%%%%%%%%%%%%%%%%%%%%%%%%%%%%%%%%%%%%%%%%%%%%%%%%%%%%%%%%%%%%

\section{The $\mathcal{I}'_\Phi$-curvatures}\label{I-prime-curv}
The CR invariance of the total $Q'$-curvature is explained by its transformation formula under changes of pseudo-Einstein contact forms (\cite{CY, H, Mar2}). In this section, we construct a Tanaka--Webster curvature quantity $\mathcal{I}'_\Phi$ for each invariant polynomial $\Phi$ of degree $n$ which integrates to the global CR invariant given by Theorem \ref{integral-inv}.  When $n=1$, the invariant is trivial, so we assume $n\ge 2$.
We denote the weighted contact form $\bth$ and the weighted Levi form $\bh_{\a\ol\b}$ simply by $\th$ and $h_{\a\ol\b}$. 

\subsection{The Lefschetz decomposition of differential forms}
Here we summarize the formulas for differential forms that we use in the subsequent subsections.

Let $V$ be an $n$-dimensional complex vector space, and let 
\[
\omega=i h_{\a\ol \b}\theta^{\a}\wedge\theta^{\ol \b}\ \in\wedge^{1, 1} V^*
\]
be a real $(1, 1)$-form on $V$, where $h_{\a\ol \b}$ is a non-degenerate hermitian form. We define the operator $L$ by $L\varphi:=\omega\wedge\varphi$. Then, for
\begin{equation}\label{phi}
\varphi=\frac{1}{p!q!}\varphi_{\a_1\cdots \a_p \ol \b_1\cdots \ol \b_q}\theta^{\a_1}\wedge\cdots
\wedge\theta^{\a_p}\wedge\theta^{\ol \b_1}\wedge\cdots\wedge\theta^{\ol \b_q}\ \in\wedge^{p, q}V^*,
\end{equation}
we have
\[
L\varphi=\frac{1}{(p+1)!(q+1)!}(L\varphi)_{\a \a_1\cdots \a_p \ol \b\ol \b_1\cdots \ol \b_q}
\theta^{\a}\wedge\theta^{\a_1}\wedge\cdots
\wedge\theta^{\a_p}\wedge\theta^{\ol \b}\wedge\theta^{\ol \b_1}\wedge\cdots\wedge\theta^{\ol \b_q},
\]
where
\[
(L\varphi)_{\a \a_1\cdots \a_p \ol \b\ol \b_1\cdots \ol \b_q}=
(-1)^p i(p+1)(q+1)h_{[\a\ol \b}\varphi_{\a_1\cdots \a_p \ol \b_1\cdots \ol \b_q]}.
\]
Let $\Lambda: \wedge^{p, q}\to \wedge^{p-1, q-1}$ be the adjoint operator of $L$ with respect to the hermitian inner product
\[
\langle\varphi, \psi\rangle_h:=\frac{1}{p!q!}\varphi_{\a_1\cdots \a_p \ol \b_1\cdots \ol \b_q}
{\ol{\psi}}^{\,\a_1\cdots \a_p \ol \b_1\cdots \ol \b_q}.
\]
For $\varphi$ given by \eqref{phi}, we have
\[
\Lambda\varphi
=\frac{-i(-1)^{p-1}}{(p-1)!(q-1)!}h^{\a\ol \b}\varphi_{\a \a_2\cdots \a_p \ol \b\ol \b_2\cdots \ol \b_q}
\theta^{\a_2}\wedge\cdots
\wedge\theta^{\a_p}\wedge\theta^{\ol \b_2}\wedge\cdots\wedge\theta^{\ol \b_q}.
\]
We note that if we write $\varphi\in\wedge^{n, n}V^*$ as
\[
\frac{1}{n!n!}\varphi_{\a_1\ol \b_1\cdots \a_n \ol b_n}\theta^{\a_1}\wedge\theta^{\ol \b_1}\wedge\cdots\wedge\theta^{\a_n}\wedge\theta^{\ol \b_n},
\]
then we have
\[
\Lambda^n \varphi=(-i)^n h^{\a_1\ol \b_1}\cdots h^{\a_n\ol \b_n}\varphi_{\a_1\ol \b_1\cdots \a_n \ol \b_n}.
\]
Similarly, for
\[
\varphi=\frac{1}{n!(n-1)!}\varphi_{\a \a_1\ol \b_1\cdots \a_{n-1} \ol \b_{n-1}}\theta^{\a}\wedge\theta^{\a_1}\wedge\theta^{\ol \b_1}\wedge\cdots\wedge\theta^{\a_{n-1}}\wedge\theta^{\ol \b_{n-1}}\in \wedge^{n, n-1}V^*,
\]
we have
\[
\Lambda^{n-1}\varphi=(-i)^{n-1} h^{\a_1\ol \b_1}\cdots h^{\a_{n-1}\ol \b_{n-1}}\varphi_{\a \a_1\ol \b_1\cdots \a_{n-1} \ol \b_{n-1}}\theta^\a.
\]

By a simple induction, we have the following proposition:
\begin{prop}
Let $\varphi\in\wedge^{p, q}V^*$. Then we have
\begin{align}
\label{Lambda-L-m}
[\Lambda, L^m]\varphi&=m(n-p-q-m+1)L^{m-1}\varphi, \\
\label{Lambda-m-L}
[\Lambda^m, L]\varphi&=m(n-p-q+m-1)\Lambda^{m-1}\varphi.
\end{align}
If $\varphi$ satisfies $\Lambda\varphi=0$, then we have
\begin{equation}\label{Lambda-k-L-m}
\Lambda^k L^m\varphi=\frac{m!}{(m-k)!}(n-p-q-m+1)\cdots(n-p-q-m+k)L^{m-k}\varphi
\end{equation}
for $k\le m$. 
\end{prop}

Thus, if we set $H\varphi:=(n-p-q)\varphi$, then $\{L, H, \Lambda\}$ forms an $\mathfrak{sl}(2)$-triple:
\[
[H, L]=-2L, \quad [H, \Lambda]=2\Lambda, \quad [\Lambda, L]=H.
\]
Consequently, $\wedge V^*$ admits the irreducible decomposition with respect to the action of $\mathfrak{sl}(2, \mathbb{C})$, which is called the {\it Lefschetz decomposition}. Using this decomposition and equations \eqref{Lambda-L-m}, \eqref{Lambda-m-L}, we can derive the following formulas:
\begin{prop}\label{diff-form}
{\rm (i)} For $\varphi\in\wedge^{n, n} V^*$, we have
\[
\varphi=\frac{1}{(n!)^2}L^n\Lambda^n\varphi.
\]
{\rm (ii)} For $\varphi\in\wedge^{n, n-1}V^*$, we have
\[
\varphi=\frac{1}{\bigl((n-1)!\bigr)^2}L^{n-1}\Lambda^{n-1}\varphi.
\]
{\rm (iii)} For $\varphi\in\wedge^{n-1, n-1}V^*$, we have
\[
L\varphi=\frac{1}{n!(n-1)!}L^n\Lambda^{n-1}\varphi.
\]
{\rm (iv)} For $\varphi\in\wedge^{m, m}V^*\ (0\le m\le n-1)$, we have
\[
\Lambda^{n-2}L^{n-m-1}\varphi
=\frac{(n-2)!(n-m-1)!}{m!}\Bigl(m\Lambda^{m-1}\varphi+(n-m-1)L\Lambda^m \varphi\Bigr).
\]
{\rm (v)} Let $n\ge 3$. For $\varphi\in \wedge^{m, m-1}V^*\ (2\le m\le n-1)$, we have
\[
\Lambda^{n-3}L^{n-m-1}\varphi=
\frac{(n-3)!(n-m-1)!}{(m-1)!}\Bigl((m-1)\Lambda^{m-2}\varphi+(n-m-1)L\Lambda^{m-1}\varphi\Bigr).
\]
\end{prop}
%%%%%%%%%%%%%%%%%%%%%%%%%%%%%%%%%%%%%%%%%%%%%%%%%%%%%%%%%%%%%%%%%%%%%%
\subsection{Definition of the $\mathcal{I}'_{\Phi}$-curvatures}
Let $(M, H, J)$ be a $(2n+1)$-dimensional compact strictly pseudoconvex CR manifold. We first define a pseudo-hermitian invariant $\mathcal{I}'_{\Phi}\in\calE(-m-1, -m-1)$ for each invariant polynomial $\Phi$ of degree $m$, and later restrict to the case $m=n$. It suffices to consider $\Phi$ of the form
\[
\Phi=T_{m_1}T_{m_2}\cdots T_{m_k}, \quad m_1+\cdots+m_k=m,
\]
where $T_j$ is defined by \eqref{def-T}. First, for each $p\ge1$, we set
\begin{align*}
\wt R^{(p)}_{A_1\ol B_1\cdots A_p \ol B_p}:=\Omega_{\g_1}{}^{\g_2}{}_{A_1\ol B_1} \Omega_{\g_2}{}^{\g_3}{}_{A_2\ol B_2}
\cdots\Omega_{\g_p}{}^{\g_1}{}_{A_p\ol B_p}\in\calE_{A_1\ol B_1\cdots A_p \ol B_p}(-p, -p),
\end{align*} 
where $\Omega_{\g\ol\mu A}{}^B$ is the curvature of the CR tractor connection.
Then, for the above $\Phi$, we define 
\begin{align*}
S^\Phi_{A_1\ol B_1} 
&:=h^{A_2\ol B_2}\cdots h^{A_m \ol B_m}
\wt R^{(m_1)}_{[A_1\ol B_1\cdots A_{m_1}\ol B_{m_1}}
\wt R^{(m_2)}_{A_{m_1+1}\ol B_{m_1+1}\cdots\ }\cdots\ 
\wt R^{(m_k)}_{\quad \cdots A_{m}\ol B_{m}]} \\
&=\begin{pmatrix}
0 & 0 & 0 \\
S^\Phi_{\a_1\ol\infty} & S^\Phi_{\a_1\ol\b_1} & 0 \\
S^\Phi_{\infty\ol\infty} & S^\Phi_{\infty\ol\b_1} & 0
\end{pmatrix}
\in\calE_{A_1\ol B_1}(-m, -m).
\end{align*}
Since $S^\Phi_{A\ol B}$ is hermitian, we have
\[
S^\Phi_{\a\ol\b}=\ol{S^\Phi_{\b\ol\a}}, \quad S^\Phi_{\a\ol\infty}=
\ol{S^\Phi_{\infty\ol\a}}, \quad S^\Phi_{\infty\ol\infty}=\ol{S^\Phi_{\infty\ol\infty}}.
\]
The real density
\[
S^\Phi:=h^{A\ol B}S^\Phi_{A\ol B}=\bh^{\a\ol\b}S^\Phi_{\a\ol\b}\in\calE(-m, -m)
\]
is CR invariant and a polynomial in the Chern--Moser tensor. Finally, using the CR $D$-operator, we define 
\begin{align*}
 \wt S^\Phi_{A\ol B}&:= S^\Phi_{A\ol B}-\frac{1}{n}S^\Phi h_{A\ol B}-\frac{1}{nm(n-2m)}(Z_{\ol B}D_A S^\Phi+Z_A D_{\ol B}S^\Phi)  \\
&=\begin{pmatrix}
0 & 0 & 0 \\
X^\Phi_\a & S^\Phi_{(\a\ol\b)_0} & 0 \\
\mathcal{I}'_\Phi & X^\Phi_{\ol\b} & 0
\end{pmatrix} \in\calE_{A\ol B}(-m, -m),
\end{align*}
where we have set 
\begin{align*}
X^\Phi_\a&:=S^\Phi_{\a\ol\infty}-\frac{1}{nm}\nabla_\a S^\Phi & &\in\calE_\a(-m, -m), \\
\mathcal{I}'_\Phi&:=S^\Phi_{\infty\ol\infty}-\frac{1}{nm(n-2m)}(\Delta_b S^\Phi+2m PS^\Phi)
& &\in\calE(-m-1, -m-1),
\end{align*}
and $X^\Phi_{\ol\a}:=\ol{X^\Phi_\a}$.
Let $\wh\theta=e^{\U}\theta$ be a rescaling of the contact form. By using \eqref{Z-W-Y},
we obtain the transformation formulas
\begin{equation}\label{hat-X-I}
\begin{aligned}
\wh X^\Phi_\a&=X^\Phi_\a-S^\Phi_{(\a\ol\b)_0}\U^{\ol\b}, \\
\wh {\mathcal{I}}'_\Phi&=\mathcal{I}'_\Phi - X^\Phi_\a \U^\a-X^\Phi_{\ol\a}\U^{\ol\a}
+S^\Phi_{(\a\ol\b)_0}\U^\a\U^{\ol\b}.
\end{aligned}
\end{equation}
\begin{rem}
By the correspondence between the CR tractor bundle and the ambient metric, we can also write as 
\begin{align*}
\wt R^{(p)}_{A_1\ol B_1\cdots A_p \ol B_p}&=
\wt R_{C_1}{}^{C_2}{}_{A_1\ol B_1} \wt R_{C_2}{}^{C_3}{}_{A_2\ol B_2}
\cdots\wt R_{C_p}{}^{C_1}{}_{A_p\ol B_p} \\
&=S_{C_1}{}^{C_2}{}_{A_1\ol B_1}S_{C_2}{}^{C_3}{}_{A_2\ol B_2}
\cdots S_{C_p}{}^{C_1}{}_{A_p\ol B_p},
\end{align*}
where $\wt R_{A\ol B C\ol E}$ is the ambient curvature tensor and $S_{A\ol B C\ol E}$ is the CR Weyl tractor; see Proposition \ref{R-S}.
\end{rem}
%%%%%%%%%%%%%%%%%%%%%%%%%%%%%%%%%%%%%%%%%%%%%%%%%%%%%%%%%%%%%%%%%%%%%%%%%%%
\subsection{Relation to the characteristic class $\Phi(T^{1, 0}M)$}
Now we specialize to the case of $m=n$, where we have
\begin{align*}
X^\Phi_\a&=S^\Phi_{\a\ol\infty}-\frac{1}{n^2}\nabla_\a S^\Phi, \\
\mathcal{I}'_\Phi&=S^\Phi_{\infty\ol\infty}+\frac{2}{n^2}PS^\Phi+\frac{1}{n^3}\Delta_b S^\Phi.
\end{align*}
If we set
\[
S^{(p)}_{\a_1\ol\b_1\cdots\a_p\ol\b_p}:=
S_{\g_1}{}^{\g_2}{}_{\a_1\ol\b_1}S_{\g_2}{}^{\g_3}{}_{\a_2\ol\b_2}\cdots S_{\g_p}{}^{\g_1}{}_{\a_p\ol\b_p},
\]
then 
\[
S^\Phi_{\a_1\ol\b_1}=\bh^{\a_2\ol\b_2}\cdots \bh^{\a_n\ol\b_n}
S^{(m_1)}_{[\a_1\ol\b_1\cdots\a_{m_1}\ol\b_{m_1}}S^{(m_2)}_{\a_{m_1+1}\ol\b_{m_1+1}\cdots\ }
\cdots \ S^{(m_k)}_{\quad \cdots \a_n\ol\b_n]}.
\]
Since $\wedge^{n, n}(T^{1, 0}M)^*\subset {\rm Im}\, L^n$ with $L:=d\theta\wedge(\cdot)$ by the Lefschetz decomposition, we have
\[
S^\Phi_{(\a\ol\b)_0}=0.
\]
Therefore, by \eqref{hat-X-I} $X^\Phi_\a$ is CR invariant, and $\mathcal{I}'_\Phi$ transforms as
\begin{equation}\label{hat-I-n}
\wh {\mathcal{I}}'_\Phi=\mathcal{I}'_\Phi - X^\Phi_\a \U^\a-X^\Phi_{\ol\a}\U^{\ol\a}.
\end{equation}
We define 
\[
(VS^{(p)})_{\a \a_1\ol\b_1\cdots \a_p\ol\b_p}
:=V_{\g_0}{}^{\g_1}{}_\a S_{\g_1}{}^{\g_2}{}_{\a_1\ol\b_1}\cdots S_{\g_p}{}^{\g_0}{}_{\a_p\ol\b_p}.
\]
Then from $\Omega_\g{}^\mu{}_{\a\ol\infty}=i V_\g{}^\mu{}_\a$, we get
\[
\wt R^{(p)}_{\a_1\ol\infty\a_2\ol\b_2\cdots \a_p\ol\b_p}=i(VS^{(p-1)})_{\a_1\a_2\ol\b_2\cdots \a_p\ol\b_p}.
\]
To simplify the presentation, we set
\[
\underline{m}_p:=m_1+m_2+\cdots+m_{p-1}+1
\]
and define $S'^{(p)}_\a$ by
\begin{align*}
S'^{(p)}_{\a_{\underline{m}_p}}&:=
\bh^{\a_1\ol\b_1}\cdots  \bh^{\a_{\underline{m}_p-1}\ol\b_{\underline{m}_p-1}}\bh^{\a_{\underline{m}_p+1}\ol\b_{\underline{m}_p+1}}
\cdots \bh^{\a_n\ol\b_n}  \\
&\quad\quad \cdot S^{(m_1)}_{[\a_1\ol\b_1\cdots}\ \cdots (VS^{(m_p-1)})_{\a_{\underline{m}_p}\cdots}\cdots S^{(m_k)}_{\quad \cdots \a_n\ol\b_n]}.
\end{align*}
Then we have
\begin{align*}
S^\Phi_{\a\ol\infty}
&= \sum_{p=1}^{k}\Bigl(\frac{m_p}{n}\bh^{\a_2\ol\b_2}\cdots \bh^{\a_n\ol\b_n}
S^{(m_1)}_{[\a_{\underline{m}_p}\ol\b_{\underline{m}_p}\a_2\ol\b_2\cdots}S^{(m_2)}_{\a_{m_1+1}\ol\b_{m_1+1}\cdots} \\
&\quad\quad\quad\quad\cdot  \cdots i(VS^{(m_p-1)})_{\a \a_{\underline{m}_p+1}\cdots}\cdots S^{(m_k)}_{\quad\cdots\a_n\ol\b_n]} \Bigr)\\
&=\frac{i}{n}\sum_{p=1}^{k} m_p S'^{(p)}_\a.
\end{align*}
Thus, $X^\Phi_\a$ is written as
\begin{equation}\label{X}
X^\Phi_\a=\frac{i}{n}\sum_{p=1}^{k}m_pS'^{(p)}_\a-\frac{1}{n^2}\nabla_\a S^\Phi.
\end{equation}

We recall that if $M$ admits a pseudo-Einstein contact form, then the first Chern class of $T^{1, 0}M$ vanishes in the real cohomology {\cite[Proposition D]{Lee}}. 
In this case, $X^\Phi_\a$ is related to the characteristic class $\Phi(T^{1,0}M)$ as follows:
\begin{prop}\label{X-Phi}
If $c_1(T^{1,0}M)=0$ in $H^2(M; \mathbb{R})$, then the CR invariant $2n$-form
\begin{equation}\label{X-form}
n^2(X^\Phi_\a\theta^\a+X^\Phi_{\ol\a}\theta^{\ol\a})\wedge\theta\wedge(d\theta)^{n-1}
\end{equation}
is a representative of the characteristic class $\Phi(T^{1, 0}M)\in H^{2n}(M; \mathbb{R})$.
\end{prop}
\begin{proof}
We consider the weighted tractor bundle
\[
\calE^A(1, 0)\cong \calE(1, 1)\oplus \calE^\a \oplus \calE(0, 0),
\]
which can always be defined globally over $M$.
Since $\calE(1, 1)$ and $\calE(0, 0)$ are trivial, we have $\Phi(T^{1, 0}M)=\Phi(\calE^A(1, 0 ))$. 
The curvature form of the connection on $\calE^A(1, 0)$ is given by
\[
\Omega_A{}^B+\frac{1}{n+2} (d\omega_\g{}^\g)\d_A{}^B,
\]
where $\omega_\a{}^\b$ is the Tanaka--Webster connection form and 
\[
\Omega_A{}^B=
\begin{pmatrix}
0 & 0 & 0 \\
* & S_\a{}^\b{}_{\g\ol\mu}\theta^\g\wedge\theta^{\ol\mu}
+V_\a{}^\b{}_\g\theta^\g\wedge\theta+V^\b{}_{\a\ol\mu}\theta\wedge\theta^{\ol\mu} & 0 \\
* & * & 0 
\end{pmatrix}
\]
is the tractor curvature form. Since $d\omega_\g{}^\g$ is exact by the assumption $c_1(T^{1, 0}M)=0$, we have
\[
\Phi(\calE^A(1, 0))=[\Phi(\Omega)].
\] 
By using Proposition \ref{diff-form} (i), (ii), we can compute $\Phi(\Omega)$ as
\[
\Phi(\Omega)=\Phi_0+\Phi_1+\ol{\Phi_1},
\]
where
\begin{align*}
\Phi_0&=i^n S^{(m_1)}_{[\a_1\ol\b_1\cdots\a_{m_1}\ol\b_{m_1}}\cdots 
S^{(m_k)}_{\quad \cdots\a_n\ol\b_n]}\theta^{\a_1}\wedge\theta^{\ol\b_1}\cdots\theta^{\a_n}\wedge\theta^{\ol\b_n} \\
&=S^{\Phi}(d\theta)^n, \\
\Phi_1&=
i^n m_1(VS^{(m_1-1)})_{[\a_1\a_2\ol\b_2\cdots}\ S^{(m_2)}_{\a_{m_1+1}\ol\b_{m_1+1}\cdots}
\cdots S^{(m_k)}_{\quad \cdots \a_n\ol\b_n]} \\
&\hspace{5cm} \cdot(\theta^{\a_1}\wedge\theta)\wedge\theta^{\a_2}\wedge\theta^{\ol\b_2}\wedge\cdots\wedge\theta^{\a_n}\wedge\theta^{\ol\b_n} \\
&\quad +\cdots \\
&\quad +
i^n m_k S^{(m_1)}_{[\a_1\ol\b_1\cdots}\ S^{(m_2)}_{\a_{m_1+1}\ol\b_{m_1+1}\cdots}
\ \cdots (VS^{(m_k-1)})_{\a_{n-m_k+1}\a_{n-m_k+2}\ol\b_{n-m_k+2}\cdots \a_n\ol\b_n]} \\
&\hspace{4cm}\cdot\theta^{\a_1}\wedge\theta^{\ol\b_1}\wedge\cdots\wedge(\theta^{\a_{n-m_k+1}}\wedge\theta)\wedge\cdots\wedge\theta^{\a_n}\wedge\theta^{\ol\b_n} \\
&=in\sum_{p=1}^k m_p S'^{(p)}_\a\theta^\a\wedge\theta\wedge(d\theta)^{n-1}.
\end{align*}
Hence by \eqref{X}, we obtain 
\begin{align*}
\Phi(\Omega)&=S^\Phi(d\theta)^n+
in\sum_{p=1}^k m_p(S'^{(p)}_\a\theta^\a-S'^{(p)}_{\ol\a}\theta^{\ol\a})\wedge\theta\wedge(d\theta)^{n-1} \\
&=n^2(X^\Phi_\a\theta^\a+X^\Phi_{\ol\a}\theta^{\ol\a})\wedge\theta\wedge(d\theta)^{n-1}
+d\bigl(S^\Phi\theta\wedge(d\theta)^{n-1}\bigr),
\end{align*}
which completes the proof.
\end{proof}

%%%%%%%%%%%%%%%%%%%%%%%%%%%%%%%%%%%%%%%%%%%%%%%%%%%%%%%%%%%%%%%%%%%%%%%%%

\subsection{CR invariance of the total $\mathcal{I}'_\Phi$-curvature}
By using the vanishing of the cohomology $\Phi(T^{1, 0}M)$, we can prove that when $\th$ is a pseudo-Einstein contact form, the integral of $\mathcal{I}'_\Phi$ is independent of the choice of $\th$:
\begin{thm}\label{invariance-I-prime}
Let $\Phi$ be an invariant polynomial of degree $n$. For pseudo-Einstein contact forms $\th$ and $\wh\th$, we have
\[
\int_M \wh{\mathcal{I}}'_\Phi=\int_M \mathcal{I}'_\Phi.
\]
\end{thm}
\begin{proof}
We can write $\wh\theta=e^{\U}\theta$ with a CR pluriharmonic function $\U$. If we set
\[
\eta:=i(\U_\a\theta^\a-\U_{\ol\a}\theta^{\ol\a})-\frac{1}{n}(\Delta_b \U) \theta,
\]
then we have $d\eta=0$ by Proposition \ref{cr-pluri}, and
\begin{align*}
\eta\wedge(X^\Phi_\a\theta^\a+X^\Phi_{\ol\a}\theta^{\ol\a})\wedge\theta\wedge(d\theta)^{n-1}
&=i(\U_\a X^\Phi_{\ol\b}+\U_{\ol\b}X^\Phi_{\a})\theta^\a\wedge\theta^{\ol\b}\wedge\theta\wedge(d\theta)^{n-1} \\
&=\frac{1}{n}(X^\Phi_\a\U^\a+X^\Phi_{\ol\a}\U^{\ol\a})\theta\wedge(d\theta)^n.
\end{align*} 
By Theorem \ref{chern-cr} and Proposition \ref{X-Phi}, this is an exact form on $M$. Hence we obtain
\[
\int_M(\wh{\mathcal{I}}'_\Phi-{\mathcal{I}}'_\Phi)=
-\int_M (X^\Phi_\a\U^\a+X^\Phi_{\ol\a}\U^{\ol\a})\theta\wedge(d\theta)^n=0.
\]
\end{proof}

\begin{rem}
If we realize $M$ as the boundary of a domain $\Omega$ in a K\"ahler manifold as in Remark \ref{rem-kahler}, we can give an alternative proof to 
the CR invariance of $\ol{\mathcal{I}}'_\Phi$ without using Theorem \ref{chern-cr} as follows:

We extend $\U$ to a pluriharmonic function $\wt\U$ on $\ol\Omega$ and set 
$\wt\eta:=i(\pa\wt\U-\ol\pa\wt\U)$. Then, $d\wt\eta=0$ and $\eta=\wt\eta\,|_{TM}$.
Since $T^{1, 0}X|_M$ is isomorphic to the direct sum of $T^{1, 0}M$ and a trivial line bundle, we have $\Phi(T^{1, 0}M)=\Phi(T^{1, 0}X)|_{TM}$. 
Therefore, by Proposition \ref{X-Phi}, 
\begin{align*}
\int_M(\wh{\mathcal{I}}'_\Phi-{\mathcal{I}}'_\Phi)&=
-\int_M (X^\Phi_\a\U^\a+X^\Phi_{\ol\a}\U^{\ol\a})\theta\wedge(d\theta)^n \\
&=-\frac{1}{n}\int_M \wt\eta\wedge \Phi(T^{1, 0}X) \\
&=-\frac{1}{n}\int_\Omega d\bigl(\wt\eta\wedge \Phi(T^{1, 0}X)\bigr) \\
&=0. 
\end{align*}
\end{rem}
\begin{rem}
The de Rham cohomology of $M$ is canonically isomorphic to the cohomology of the Rumin complex (\cite{R}). For an invariant polynomial $\Phi$ of degree $n$, exactness of the $2n$-form 
\eqref{X-form} implies that one can write as
\[
X^\Phi_\a=\nabla^{\ol\b}\omega_{\a\ol\b}+i\nabla^\b \omega_{\a\b}
\]
with a trace-free hermitian tensor $\omega_{\a\ol\b}\in\calE_{(\a\ol\b)_0}(-n+1, -n+1)$ and a skew symmetric tensor $\omega_{\a\b}\in\calE_{[\a\b]}(-n+1, -n+1)$. It follows that 
\[
{\rm Re}(X^\Phi_\a \U^\a)=P_1 \U+P_2\U,
\]
where 
\[
P_1\U :={\rm Re}((\nabla^{\ol\b}\omega_{\a\ol\b})\U^\a), \quad   
P_2\U:= {\rm Re} ((i\nabla^\b \omega_{\a\b})\U^\a).
\]
For CR pluriharmonic functions $\U, \U'$, we have
\[
\int_M (\U P_1\U'-\U' P_1\U)=\int_M (\U P_2\U'+\U' P_2 \U)=0,
\]
which explains the CR invariance of $\ol{\mathcal{I}}'_\Phi$. As is pointed out in 
{\cite[Remark 8.12]{CG}}, the operator $\U\mapsto {\rm Re}(X^\Phi_\a \U^\a)$ is formally self-adjoint on the space of CR pluriharmonic functions if we can take $\omega_{\a\b}=0$.

\end{rem}

%%%%%%%%%%%%%%%%%%%%%%%%%%%%%%%%%%%%%%%%%%%%%%%%%%%%%%%%%%%%%%%%%%%%%%%%%%%
\subsection{The $\mathcal{I}_\Phi$-invariant}
We can define a CR invariant density 
$\mathcal{I}_\Phi$, for which $\mathcal{I}'_\Phi$ can be considered as the ``primed analogue'' (cf. {\cite[Remark 8.11]{CG}}):

\begin{prop}
For an invariant polynomial $\Phi$ of degree $m\ge 2$, define a density 
$\mathcal{I}_\Phi\in\calE(-m-1, -m-1)$ by
\begin{align*}
\mathcal{I}_\Phi&:=\nabla^\a X^\Phi_\a-S^\Phi_{\a\ol\b}P^{(\ol\b\a)_0}+(n-m)\mathcal{I}'_\Phi \\
&\,=\nabla^\a\Bigl(S^\Phi_{\a\ol\infty}-\frac{1}{nm}\nabla_\a S^\Phi\Bigr)-S^\Phi_{\a\ol\b}P^{(\ol\b\a)_0} \\
&\quad +(n-m)
\Bigl(S^\Phi_{\infty\ol\infty}-\frac{2}{n(n-2m)}PS^\Phi-\frac{1}{nm(n-2m)}\Delta_b S^\Phi\Bigr).
\end{align*}
Then we have
\[
D^A \wt S^\Phi_{A\ol B}=(n-2m+1)\mathcal{I}_\Phi Z_{\ol B}.
\]
In particular, $\mathcal{I}_\Phi$ is a CR invariant.
\end{prop}
\begin{rem}
Since any tractor $t\in\calE_*(w, w')$ satisfies
\[
D^{\ol B}(t Z_{\ol B})=(n+w+1)(n+w+w'+2)t,
\]
we also have the equation
\[
D^A D^{\ol B}\wt S^\Phi_{A\ol B}= D^{\ol B}D^A\wt S^\Phi_{A\ol B}
=(n-2m+1)(n-m)(n-2m)\mathcal{I}_\Phi.
\]
\end{rem}
\bigskip

By a direct computation, one has
\[
D^A \wt S^\Phi_{A\ol B}=W_{\ol B}^{\ol\b} F_{\ol\b}+(n-2m+1)\mathcal{I}_\Phi Z_{\ol B},
\]
with
\[
F_{\ol\b}=\frac{n-2m+1}{m}\bigl(m(n-m)S^\Phi_{\infty\ol\b}+m\nabla^\a S^\Phi_{\a\ol\b}-\nabla_{\ol\b}S^\Phi\bigr).
\]
Hence it suffices to prove $F_{\ol\b}=0$. When $m=n$ this follows from $S^{\Phi}_{(\a\ol\b)_0}=0$, so we assume that $2\le m\le n-1$. 

We fix a contact form $\theta$ and consider a $2m$-form 
\[
\varphi=\varphi^{(0)}+\theta\wedge(\varphi^{(1)}+\varphi^{(2)})
\]
on $M$, where 
\begin{align*}
\varphi^{(0)}&=\varphi^{(0)}_{\a_1\ol\b_1\cdots \a_m\ol\b_m}\theta^{\a_1}\wedge\theta^{\ol\b_1}\wedge\cdots\wedge\theta^{\a_m}\wedge\theta^{\ol\b_m} & &\in \wedge^{m, m}(T^{1, 0}M)^*,\\
\varphi^{(1)}&=\varphi^{(1)}_{\a\a_2\ol\b_2\cdots \a_m\ol\b_m}\theta^\a\wedge\theta^{\a_2}\wedge\theta^{\ol\b_2}\wedge\cdots\wedge\theta^{\a_m}\wedge\theta^{\ol\b_m} & &\in \wedge^{m, m-1}(T^{1,0}M)^*, \\
\varphi^{(2)}&=\varphi^{(2)}_{\ol\b \a_2\ol\b_2\cdots \a_m\ol\b_m}\theta^{\ol\b}\wedge\theta^{\a_2}\wedge\theta^{\ol\b_2}\wedge\cdots\wedge\theta^{\a_m}\wedge\theta^{\ol\b_m} & &\in \wedge^{m-1, m}(T^{1, 0}M)^*.
\end{align*}
We set
\begin{align*}
\tilde\varphi^{(0)}_{\a_1\ol\b_1}&:=h^{\a_2\ol\b_2}\cdots h^{\a_m\ol\b_m}\varphi^{(0)}_{\a_1\ol\b_1\cdots \a_m\ol\b_m}, & 
\hat\varphi^{(0)}&:=h^{\a_1\ol\b_1}\tilde\varphi^{(0)}_{\a_1\ol\b_1}, \\
\tilde\varphi^{(1)}_{\a \a_2\ol\b_2}&:=h^{\a_3\ol\b_3}\cdots h^{\a_m\ol\b_m}\varphi^{(1)}_{\a\a_2\ol\b_2\cdots \a_m\ol\b_m}, & 
\hat\varphi^{(1)}_{\a}&:=h^{\a_2\ol\b_2}\tilde\varphi^{(1)}_{\a \a_2\ol\b_2}, \\
\tilde\varphi^{(2)}_{\ol\b \a_2\ol\b_2}&:=h^{\a_3\ol\b_3}\cdots h^{\a_m\ol\b_m}\varphi^{(2)}_{\ol\b\a_2\ol\b_2\cdots \a_m\ol\b_m}, & 
\hat\varphi^{(2)}_{\ol\b}&:=h^{\a_2\ol\b_2}\tilde\varphi^{(2)}_{\ol\b \a_2\ol\b_2}.
\end{align*}
 
\begin{prop}
Suppose that $d\varphi=0$. Then we have
\begin{equation}\label{closedness}
m\nabla^{\ol\b}\tilde\varphi^{(0)}_{\a\ol\b}-\nabla_\a \hat\varphi^{(0)}-(n-m)i\hat\varphi^{(1)}_\a=0.
\end{equation}
\end{prop}
\begin{proof}
We consider the closed form
\[
L^{n-m-1}\varphi=L^{n-m-1}\varphi^{(0)}+\theta\wedge(L^{n-m-1}\varphi^{(1)}+L^{n-m-1}\varphi^{(2)}),
\]
where $L=d\theta\wedge(\cdot)$. By the Lefschetz decomposition, we can write as
\begin{equation}\label{phi0}
L^{n-m-1}\varphi^{(0)}=L^{n-1}\psi+L^{n-2}\psi_1, \qquad \psi\in\wedge^{0, 0},\ \psi_1\in\wedge_0^{1,1}:={\rm Ker}\, \Lambda.
\end{equation}
First, we apply $\Lambda^{n-1}$ to \eqref{phi0}. By using \eqref{Lambda-k-L-m} and Proposition \ref{diff-form} (iv), we obtain
\[
\psi=\frac{(n-m)!}{n!m!}\Lambda^m \varphi^{(0)}.
\]
Next, we apply $\Lambda^{n-2}$ to \eqref{phi0} and use the same formulas to obtain 
\[
\psi_1=\frac{(n-m-1)!}{(n-2)!(m-1)!}\Bigl(\Lambda^{m-1}\varphi^{(0)}-\frac{1}{n}L\Lambda^{m}\varphi^{(0)}\Bigr).
\] 
Thus we have
\[
L^{n-m-1}\varphi^{(0)}
=\frac{(n-m-1)!}{(n-2)!(m-1)!}\Bigl(L^{n-2}\Lambda^{m-1}\varphi^{(0)}-
\frac{m-1}{m(n-1)}L^{n-1}\Lambda^m \varphi^{(0)}\Bigr).
\]
In a similar way, we can write
\[
L^{n-m-1}\varphi^{(1)}=L^{n-2}\eta_1+L^{n-3}\eta_2, \qquad \eta_1\in\wedge^{1,0},\ \eta_2\in\wedge_0^{2,1}
\]
and compute $\eta_1, \eta_2$ by using Proposition \ref{diff-form} (v). As a result, we obtain
\[
L^{n-m-1}\varphi^{(1)}=\frac{(n-m-1)!}{(n-3)!(m-2)!}\Bigl(L^{n-3}\Lambda^{m-2}\varphi^{(1)}-\frac{m-2}{(n-2)(m-1)}L^{n-2}\Lambda^{m-1}\varphi^{(1)}\Bigr).
\]
The same formula holds for $\varphi^{(2)}$. It follows that
\begin{align*}
\frac{(n-2)!(m-1)!}{(n-m-1)!}L^{n-m-1}\varphi
&=L^{n-2}\Lambda^{m-1}\varphi^{(0)}-\frac{m-1}{m(n-1)}L^{n-1}\Lambda^{m}\varphi^{(0)} \\
&\quad +\theta\wedge\bigl((n-2)(m-1)L^{n-3}\Lambda^{m-2}(\varphi^{(1)}+\varphi^{(2)}) \\
&\quad\quad\quad\quad-(m-2)L^{n-2}\Lambda^{m-1}(\varphi^{(1)}+\varphi^{(2)})\bigr).
\end{align*}
The right-hand side is a non-zero multiple of 
\begin{align*}
&\tilde\varphi^{(0)}_{\a\ol\b}\theta^\a\wedge\theta^{\ol\b}\wedge(d\theta)^{n-2}+\frac{m-1}{m(n-1)}i\hat\varphi^{(0)}(d\theta)^{n-1} \\
&\quad  +\theta\wedge\Bigl(\frac{(n-2)(m-1)}{2m}i(\tilde\varphi^{(1)}_{\g\a\ol\b}\theta^\g\wedge\theta^\a\wedge\theta^{\ol\b}+\tilde\varphi^{(2)}_{\ol\g\a\ol\b}\theta^{\ol\g}\wedge\theta^\a\wedge\theta^{\ol\b})\wedge(d\theta)^{n-3} \\
&\qquad\qquad\ -\frac{m-2}{m}(\hat\varphi^{(1)}_\a\theta^\a+\hat\varphi^{(2)}_{\ol\a}\theta^{\ol\a})\wedge(d\theta)^{n-2}\Bigr)
\end{align*}
and this must a be closed form. We differentiate this form, computing modulo $\theta$, and look at the coefficient of $\theta^{\g}\wedge\theta^\a\wedge\theta^{\ol\b}\wedge(d\theta)^{n-2}$. Then we have
\[
\nabla_{[\g}\tilde\varphi^{(0)}_{\a]\ol\b}-\frac{m-1}{m(n-1)}\nabla_{[\g}\hat\varphi^{(0)}h_{\a]\ol\b} +\frac{(n-2)(m-1)}{2m}i\tilde\varphi^{(1)}_{\g\a\ol\b}-\frac{m-2}{m}i\hat\varphi^{(1)}_{[\g}h_{\a]\ol\b}=0.
\]
Taking the trace with respect to $(\g, \ol\b)$, we obtain \eqref{closedness}.
\end{proof}

Now we apply the proposition to the closed form $\varphi=\Phi(\Omega)$. Since 
\[
\tilde\varphi^{(0)}_{\a\ol\b}=S^{\Phi}_{\a\ol\b}, \quad \hat\varphi^{(0)}=S^{\Phi},\quad 
\hat\varphi^{(1)}_{\a}=imS^{\Phi}_{\a\ol\infty},
\]
the equation \eqref{closedness} implies $F_{\ol\a}=0$.

%%%%%%%%%%%%%%%%%%%%%%%%%%%%%%%%%%%%%%%%%%%%%%%%%%%%%%%%%%%%%%%%%%%%%%%%%
\subsection{Explicit formulas}
We will give some explicit formulas for $X^\Phi_\a$ and $\mathcal{I}'_\Phi$. We note that 
by $\Omega_{\a\ol\b A}{}^A=0$ we have $\mathcal{I}'_\Phi=0$ when $\Phi$ is factored by $T_1(=c_1)$. Thus, $\Phi$ should be considered modulo $c_1$. For $\Phi=T_2\ (\equiv -2c_2\ {\rm mod}\ c_1)$, we have
\begin{align*}
X^{T_2}_\a
&=-\frac{i}{2}S_{\a\ol\b\g\ol\mu}V^{\ol\b\g\ol\mu}+\frac{1}{4n}\nabla_\a|S_{\g\ol\d\mu\ol\nu}|^2, \\
\mathcal{I}'_{T_2}
&=-\frac{1}{2}|V_{\a\ol\b\g}|^2+\frac{1}{n(n-4)}P|S_{\a\ol\b\g\ol\mu}|^2+
\frac{1}{4n(n-4)}\Delta_b |S_{\a\ol\b\g\ol\mu}|^2. 
\end{align*}
When $n=2$, these coincide with $\frac{1}{2}X_\a$ and $-\frac{1}{2}\mathcal{I}'$ of Case--Gover \cite{CG}. 

For $\Phi=T_3\ (\equiv 3c_3\ {\rm mod}\ c_1)$, we have
\begin{align*}
X^{T_3}_\a
&=\frac{1}{6}(-i V_{\g_1}{}^{\g_2}{}_\a S_{\g_2}{}^{\g_3}{}_\mu{}^\nu S_{\g_3}{}^{\g_1}{}_\nu{}^\mu+i S_{\g_1}{}^{\g_2}{}_{\a}{}^\nu V_{\g_2}{}^{\g_3}{}_{\mu} S_{\g_3}{}^{\g_1}{}_\nu{}^\mu \\
&\quad\quad +i S_{\g_1}{}^{\g_2}{}_{\a}{}^\nu S_{\g_2}{}^{\g_3}{}_\nu{}^\mu V_{\g_3}{}^{\g_1}{}_{\mu }) -\frac{1}{3n}\nabla_\a S^{T_3}\\
\mathcal{I}'_{T_3}
&=\frac{1}{6}(-U_{\g_1}{}^{\g_2}S_{\g_2}{}^{\g_3}{}_\mu{}^\nu S_{\g_3}{}^{\g_1}{}_\nu{}^\mu
+V^\mu{}_{\g_1}{}^{\g_2}V_{\g_2}{}^{\g_3}{}_{\nu}S_{\g_3}{}^{\g_1}{}_\mu{}^\nu \\
&\quad\quad\ +V^\mu{}_{\g_1}{}^{\g_2}S_{\g_2}{}^{\g_3}{}_\mu{}^\nu V_{\g_3}{}^{\g_1}{}_{\nu}) 
-\frac{2}{n(n-6)}P S^{T_3} -\frac{1}{3n(n-6)}\Delta_b S^{T_3},
\end{align*}
where
\[
S^{T_3}=\frac{1}{6}(S_{\g_1}{}^{\g_2}{}_{\a}{}^\nu S_{\g_2}{}^{\g_3}{}_\nu{}^\mu S_{\g_3}{}^{\g_1}{}_{\mu}{}^\a+S_{\g_1}{}^{\g_2}{}_{\a}{}^\nu S_{\g_2}{}^{\g_3}{}_{\mu}{}^{\a} S_{\g_3}{}^{\g_1}{}_\nu{}^\mu).
\]
In general, when the degree of $\Phi$ is $m$, they are roughly of the following forms:
\begin{align*}
X^\Phi_\a&=(V\otimes S^{\otimes(m-1)})_\a+\nabla_\a S^{\otimes m}, \\
\mathcal{I}'_\Phi&=U\otimes S^{\otimes(m-1)}+V^{\otimes 2}\otimes S^{\otimes(m-2)}
+PS^{\otimes m}+\Delta_b S^{\otimes m}.
\end{align*}

%%%%%%%%%%%%%%%%%%%%%%%%%%%%%%%%%%%%%
\subsection{Computations for circle bundles over K\"ahler--Einstein manifolds}
We will compute the total $\mathcal{I}'_{\Phi}$-curvatures for the following CR manifold: 

Let $Y\subset\mathbb{CP}^{n+r}$ be a smooth projective variety given by the common zero locus $\{f_1=f_2=\cdots=f_r=0\}$, where $f_i$ is a homogeneous polynomial of degree $d_i$. We assume that $-n-r-1+\sum_{i=1}^{r} d_i>0$ so that the canonical bundle 
\[
K_Y=\mathcal{O}(-n-r-1+\textstyle\sum_{i=1}^{r}d_i)|_Y
\]
is positive. We also assume that $r>n$. By the Aubin--Yau theorem, there exists a K\"ahler metric $\bar g$ on $Y$ which satisfies 
the Einstein equation
\[
{\rm Ric}(\bar g)=-\bar g.
\]
We let $h_{K_Y}=(\det \bar g)^{-1}$ be the induced hermitian metric on $K_Y$ and consider the circle bundle
\[
M:=\{v\in K_Y\ |\ h_{K_Y}(v, v)=1\}.
\]
We take $\rho:=\log h_{K_Y}(v, v)$ as a defining function, and let $\{\pa\rho, \theta^\a\}$ be a $(1, 0)$-coframe around a point on $M$, where $\{\theta^\a\}$ is (the pull-back of) a $(1, 0)$-coframe on $Y$. Then, $\{\theta^\a\}$ gives a coframe for $T^{1, 0}M$ and we have
\begin{align*}
-\pa\ol\pa\rho&=\pa\ol\pa\log\det \bar g\\
&=-{\rm Ric}(\bar g)_{\a\ol\b}\theta^\a\wedge\theta^{\ol\b} \\
&=\bar g_{\a\ol\b}\theta^\a\wedge\theta^{\ol\b}.
\end{align*}
Thus, $M$ is a strictly pseudoconvex CR manifold with the Levi form for $\theta:=(d^c\rho)|_{TM}$ being $\bar g_{\a\ol\b}$. Moreover, the Tanaka--Webster connection form is equal to (the pull-back of) the Levi-Civita connection form of $\bar g$ and the Tanaka--Webster torsion vanishes: $A_{\a\b}=0$. Consequently, $\theta$ is a pseudo-Einstein contact form and we have $P=-n/(2(n+1))$, $V_{\a\ol\b\g}=0$, $U_{\a\ol\b}=0$, and
\begin{equation}\label{S-B}
S_{\a\ol\b\g\ol\mu}=R^{\bar g}_{\a\ol\b\g\ol\mu}+\frac{1}{n+1}
(\bar g_{\a\ol\b}\bar g_{\g\ol\mu}+\bar g_{\g\ol\b}\bar g_{\a\ol\mu})=B^{\bar g}_{\a\ol\b\g\ol\mu},
\end{equation}
where $B^{\bar g}_{\a\ol\b\g\ol\mu}$ is the Bochner tensor of $\bar g$. 

Let $\zeta=|\zeta|e^{is}$ be the fiber coordinate of $K_Y$ with respect to the local section 
$\theta^1\wedge\cdots\wedge\theta^n$. Then we have
\begin{equation}\label{volume}
\begin{aligned}
\theta\wedge(d\theta)^n&={\rm Re}(i\pa\rho)\wedge(-i\pa\ol\pa\rho)^n|_{TM} \\
&=ds\wedge\omega_{\bar g}^n,
\end{aligned}
\end{equation}
where $\omega_{\bar g}:=i\bar g_{\a\ol\b}\theta^\a\wedge\theta^{\ol\b}$ is the K\"ahler form of $\bar g$.

The integral of $\mathcal{I}'_\Phi$ over $M$ for an invariant polynomial $\Phi=T_{m_1}\cdots T_{m_k}$ of degree $n$ is computed as 
\[
\int_M \mathcal{I}'_\Phi=\int_M \frac{2}{n^2} PS^\Phi \theta\wedge(d\theta)^n
=-\frac{1}{n(n+1)}\int_M S^\Phi \theta\wedge(d\theta)^n.
\]
By \eqref{S-B} and \eqref{volume}, we have
\begin{align*}
\int_M S^\Phi \theta\wedge(d\theta)^n
&=2\pi\int_Y S^\Phi \omega_{\bar g}^n \\ 
&=2\pi\int_Y i^n S^{(m_1)}_{[\a_1\ol\b_1\cdots} \cdots S^{(m_k)}_{\quad \cdots\a_n\ol\b_n]}
\theta^{\a_1}\wedge\theta^{\ol\b_1}\wedge\cdots\wedge\theta^{\a_n}\wedge\theta^{\ol\b_n} \\
&=2\pi\int _Y \Phi(B^{\bar g}),
\end{align*}
where $B^{\bar g}:=B^{\bar g}{}_{\a}{}^\b{}_{\g\ol\mu}\theta^\g\wedge\theta^{\ol\mu}$.
Since $\bar g$ is K\"ahler--Einstein, $\Phi(B^{\bar g})$ can be rewritten to a usual characteristic form $\wt\Phi(R^{\bar g})$ as in the computation of renormalized characteristic forms. Hence, the total $\mathcal{I}'_\Phi$-curvature is represented as a characteristic number of $T^{1, 0}Y$:
\[
\int_M \mathcal{I}'_\Phi=-\frac{2\pi}{n(n+1)}\langle \wt\Phi(T^{1, 0}Y), [Y]\rangle.
\]

Let us consider the case $\Phi=c_{i_1}\cdots c_{i_k}\ (2\le i_l \le n,\ i_1+\cdots +i_k=n)$. We have 
$\wt\Phi=\wt c_{i_1}\cdots \wt c_{i_k}$, where each $\wt c_i$ is written in the form
\[
\wt c_i=c_i+\sum_{j=1}^{i} c_1^{\, j} \phi_j
\]
with invariant polynomials $\phi_j$ of degree $i-j$. The Chern classes of $T^{1, 0}Y$ can be described in terms of the data $d_1, \dots, d_r$ as follows (see \cite{Hz}):

If we set $x:= c_1(\mathcal{O}(1)|_Y)\in H^2(Y; \mathbb{R})$, then the total Chern class $c(T^{1, 0}Y)=1+c_1(T^{1. 0}Y)+\cdots+ c_n(T^{1, 0}Y)$ is given by
\begin{align*}
c(T^{1, 0}Y)&=(1+x)^{n+r+1}(1+d_1x)^{-1}\cdots(1+d_r x)^{-1} \\
&=(1+x)^{n+r+1}(1+\sigma_1 x+\cdots+\sigma_r x^r)^{-1},
\end{align*}
where $\sigma_j$ is the elementary symmetric polynomial of degree $j$, and we expand $(1+d_j x)^{-1}$ and $(1+\sigma_1 x+\cdots+\sigma_r x^r)^{-1}$ as a power series in $x$, which is actually a finite sum.
It follows that $c_i(T^{1, 0}Y)$ is $x^i$ times a symmetric polynomial in $d_1, \dots, d_r$ of degree $i$, and its homogeneous part of degree $i$ is of the form
\[
-\sigma_i+({\rm higher\ order\ terms\ in}\ \sigma_j).
\]
Therefore, by using the formula
\[
\langle x^n, [Y]\rangle=d_1\cdots d_r=\sigma_r,
\]
we obtain that $\langle \wt\Phi(T^{1, 0}Y), [Y]\rangle$ is a symmetric polynomial in $d_1, \dots, d_r$ of degree $n+r$ whose homogeneous part of degree $n+r$ is of the form
\begin{equation*}
(-1)^k\sigma_{i_1}\cdots\sigma_{i_k}\sigma_r+({\rm higher\ order\ terms\ in}\ \sigma_j).
\end{equation*}
By the assumption $r> n$, the polynomials $\sigma_i\ (2\le i\le n, i=r)$ are algebraically independent. Hence we have the following proposition:
\begin{prop}
Let $\Phi\neq0$ be an invariant polynomial of degree $n$ which can be written as a polynomial in $c_2, \dots, c_n$. Then the total $\mathcal{I}'_\Phi$-curvature $\overline{\mathcal{I}}'_\Phi$ of the circle bundle $M$ is a non-zero symmetric polynomial in $d_1, \cdots, d_r$ of degree $n+r$. Moreover, 
if $\Phi_1\neq\Phi_2$ then $\overline{\mathcal{I}}'_{\Phi_1}\neq\overline{\mathcal{I}}'_{\Phi_2}$ as polynomials.
\end{prop}

\begin{rem}\label{general-case}
By this proposition and Theorem \ref{fp-I-prime} below, the CR invariants given by Theorem \ref{integral-inv} are non-trivial when $m=n$. Takeuchi \cite{T2} computed the CR invariants for general $\Phi$ for Sasakian $\eta$-Einstein manifolds, i.e., CR manifolds which admit pseudo-Einstein  contact forms whose Tanaka--Webster torsion tensors vanish. Applying his formula to the circle bundle above, we obtain an expression of the invariant for 
\[
\Phi=c_{i_1}\cdots c_{i_k}\ (2\le i_l \le n,\ i_1+\cdots +i_k=m)
\] 
as a symmetric polynomial of degree $n+r$ in $d_1, \dots, d_r$. Up to a non-zero constant multiple, its homogeneous part of degree $n+r$ is given by 
\[
\sigma_1^{n-m}\sigma_{i_1}\cdots\sigma_{i_k}\sigma_r+({\rm higher\ order\ terms\ in}\ \sigma_j).
\]
This indicates that the invariants are non-trivial and mutually independent. 
\end{rem}
%%%%%%%%%%%%%%%%%%%%%%%%%%%%%%%%%%%%%%%%%%%%%%%%%%%%%%%%%%%%%%%%%%%%%%%%%%
\section{The total $\mathcal{I}'_\Phi$-curvatures and renormalized characteristic forms}\label{relation}

 \subsection{Ambient description of the CR tractor}\label{ambient-tractor}
We will prove that the integral of $\mathcal{I}'_\Phi$ agrees with a multiple of the global CR invariant given by Theorem \ref{integral-inv}. Since the $\mathcal{I}'_\Phi$-curvature is defined via CR tractor calculus, we first need to establish the (local) correspondence between the ambient metric and the CR tractor bundles. In {\cite[\S 5]{CG}}, the correspondence is described by using Fefferman's conformal structure as an intermediate geometry.  Here we directly give the explicit correspondence in the framework of Graham--Lee. The direct construction of CR tractors by the ambient metric is also sketched in {\cite[\S 3.4]{Cap}}.

We consider the setting of  \S \ref{spc}: Let $\Omega$ be a strictly pseudoconvex domain with the boundary $M$ in a complex manifold $X$ of dimension $n+1$. Let $\wt g=-i\pa\ol\pa\br$ be the ambient metric on $\wt X=K_X\setminus\{0\}$ for a fixed Fefferman defining density $\br\in\wt\calE(1)$, and $\wt\nabla$ its Levi-Civita connection. Recall that we defined the dilation $\d_\lambda\ (\lambda\in\C^*)$ on $\wt X$ by $\d_\lambda(u):=\lambda^{n+2}u$. 

We work locally on a neighborhood of a boundary point and fix an $(n+2)$-nd root $\wt X_0:=\wt X^{\frac{1}{n+2}}$ of $\wt X$. We define the dilation $\d^0_\lambda$ on $\wt X_0$ by $\d^0_\lambda(u):=\lambda u$ so that the covering map $\tilde\pi: \wt X_0\rightarrow \wt X$ is $\mathbb{C}^*$-equivariant. By the pull-back via $\tilde\pi$, a density $f\in\wt\calE(w)$ can be identified with a function on $\wt X_0$ which satisfies $(\d^0_\lambda)^*f=|\lambda|^{2w}f$. We also pull back the ambient metric to $\wt X_0$ and denote the metric and its Levi-Civita connection by the same symbols $\wt g$, $\wt\nabla$. 

Let $Z\in T^{1, 0}\wt X_0$ be the $(1, 0)$-part of the infinitesimal generator of $\d^0$. Then 
it can be written as $Z=z^0(\partial/\partial z^0)$, where $z^0$ is the fiber coordinate associated to a section $(dz^1\wedge\cdots \wedge dz^{n+1})^{\frac{1}{n+2}}$ with local coordinates $(z^1, \dots, z^{n+1})$. Since $z^0$ corresponds via $\tilde\pi$ to the local fiber coordinate of $\wt X$ defined in \S \ref{BE-conn}, the formulas \eqref{ambient-conn-coordinate}, \eqref{ambient-curvature-coordinate}, \eqref{omega-rho-0} for the connection and the curvature forms of $\wt\nabla$ hold true on $\wt X_0$ as well.

From $Z^A \wt g_{A\ol B}=-\wt\nabla_{\ol B}\br$, we have
\begin{equation}\label{Z-R}
\wt\nabla_A Z^B=\d_A{}^B, \quad Z^A \wt R_{A\ol B C\ol E}=0,
\end{equation}
where $\wt R$ is the curvature tensor of $\wt g$ on $\wt X_0$. A vector field $V\in \mathbb{C}T\wt X_0$ is said to be homogeneous of degree $(w, w')\in \mathbb{C}^2\ (w-w'\in \mathbb{Z})$ if it satisfies
\[
(\d^0_\lambda)_* V=\lambda^{-w}\bar\lambda^{-w'}V
\]
for all $\lambda\in\mathbb{C}^*$. 
The homogeneity is also characterized by the Lie derivatives:
\[
\mathcal{L}_Z V(=[Z, V])=w V, \quad \mathcal{L}_{\ol Z}V(=[\ol Z, V])=w' V.
\]
The following lemma relates the homogeneity to the covariant derivatives:
\begin{lem}
A $(1, 0)$-vector field $V\in T^{1, 0}\wt X_0$ satisfies
\[
\wt\nabla_Z V=\wt\nabla_{\ol Z} V=0
\]
if and only if $V$ is homogeneous of degree $(-1, 0)$.
\end{lem}

\begin{proof}
By using \eqref{Z-R} and the fact that $Z$ is a holomorphic vector field, we have
\begin{equation}\label{nabla-Z}
\wt\nabla_V Z=V, \quad \wt\nabla_{V} \ol Z=0
\end{equation}
for any $(1, 0)$-vector $V$. Thus we obtain
\[
\mathcal{L}_Z V=\wt\nabla_Z V-V, \quad \mathcal{L}_{\ol Z} V=\wt\nabla_{\ol Z} V,
\]
from which the lemma follows.
\end{proof}
We consider the set of all $(1, 0)$-vector fields along $\mathcal{N}_0=\{\br=0\}\subset\wt X_0$ which are parallel in the fiber direction, or equivalently homogeneous of degree $(-1, 0)$:
\begin{align*}
\wt {\mathcal{T}}^{1, 0}&:=\{V\in \Gamma(T^{1, 0}{\wt X_0}|_{\mathcal{N}_0})\ |\ 
\wt\nabla_Z V=\wt\nabla_{\ol Z}V=0\} \\
&=\{V\in \Gamma(T^{1, 0}{\wt X_0}|_{\mathcal{N}_0})\ |\ V\ {\rm is}\ {\rm homogeneous}\ {\rm of }\ {\rm degree}\ (-1, 0)\}.
\end{align*}
A vector field $V\in\wt {\mathcal{T}}^{1, 0}$ can be identified with a section of the rank $n+2$ complex vector bundle over $M$ defined by $\mathcal{T}^{1, 0}:=(T^{1, 0}{\wt X_0}|_{\mathcal{N}_0})/
\mathbb{C}^*$, where $\lambda\in\mathbb{C}^*$ acts as $V\mapsto \lambda^{-1}(\delta^0_{\lambda})_* V$.  The ambient metric $\wt g$ defines a hermitian metric on $\mathcal{T}^{1, 0}$ since $\wt g(V, \ol W)$ is homogeneous of degree $0$ for $V, W\in \wt{\mathcal{T}}^{1, 0}$. We will define a connection on $\mathcal{T}^{1, 0}$ and show that these agree with the CR tractor bundle and the tractor connection.

For a vector field $\nu\in TM$, let $\tilde\nu\in T\mathcal{N}_0\subset T\wt X_0$ be a lift to $\mathcal{N}_0$ which is homogeneous of degree $(0, 0)$. Then, for $V\in \wt {\mathcal{T}}^{1, 0}$ the derivative $\wt \nabla_{\tilde \nu}V$ is also parallel along the fibers since
\[
\wt\nabla_{Z}\wt\nabla_{\tilde \nu}V=\wt\nabla_{\tilde\nu}\wt\nabla_{Z}V+
\wt\nabla_{[Z, \tilde\nu]} V+\wt R(Z, \tilde \nu)V=0
\]
by \eqref{Z-R} and similarly $\wt\nabla_{\ol Z}\wt\nabla_{\tilde \nu}V=0$. Moreover, 
$\wt \nabla_{\tilde \nu}V$ does not depend on the choice of lift $\wt \nu$. Thus we can define a  linear connection $\nabla^{\mathcal{T}}$ on $\mathcal{T}^{1, 0}$ by 
\[
\nabla^{\mathcal{T}}_{\nu} V:=\wt\nabla_{\tilde\nu} V.
\]
We also define a connection on $\mathcal{T}^{1, 0}(w, w'):=\mathcal{T}^{1, 0}\otimes \calE(w, w')$ by coupling $\nabla^{\mathcal{T}}$ with the Tanaka--Webster connection on $\mathcal{E}(w, w')$ determined by a choice of CR scale. We also denote the connection by $\nabla^{\mathcal{T}}$. The relation between this coupled connection and the ambient connection $\wt\nabla$ is described as follows. 

Let $\sigma:=(dz^1\wedge\cdots\wedge dz^{n+1})^{\frac{-1}{n+2}}\in\calE(1, 0)$ be the local pseudo-Einstein CR scale associated to the coordinates. Then, a section $V\in \mathcal{T}^{1, 0}(w, w')$ is identified with the vector field 
\[
\wt V:=(z^0)^w(z^{\ol 0})^{w'} \wt{V}_0 \in T^{1, 0}{\wt X_0}|_{\mathcal{N}_0},
\]
which is homogeneous of degree $(w-1, w')$. Here, $\wt{V}_0\in\wt{\mathcal{T}}^{1, 0}$ is the vector field corresponding to the section $V_0=\sigma^{-w}\bar{\sigma}^{-w'}V\in\mathcal{T}^{1, 0}$.
For $\nu\in TM$, we take the lift $\tilde\nu$ so that it satisfies 
\[
dz^0(\tilde \nu)=dz^{\ol 0}(\tilde\nu)=0,
\]
and call it the {\it horizontal lift} with respect to $(z^1, \dots, z^{n+1})$.
Then the vector field $\wt\nabla_{\tilde\nu}\wt V$ is homogeneous of degree $(w-1, w')$, so it defines a section of $\mathcal{T}^{1, 0}(w, w')$. By definition, this agrees with $\nabla^{\mathcal{T}}_\nu V$ when $w=w'=0$. In general cases however, the differentiations in the $T$-direction may have difference:

\begin{prop}\label{tractor-conn-lift}
Let $\wt V\in \Gamma(T^{1, 0}{\wt X_0}|_{\mathcal{N}_0})$ be the homogeneous vector field of degree $(w-1, w')$ corresponding to a section $V \in\mathcal{T}^{1, 0}(w, w')$. Let $\nu\in T^{1, 0}M\oplus T^{0, 1}M$ and let $T$ be the Reeb vector field for the CR scale $|\sigma|^2$. Then we have 
\begin{align*}
\wt\nabla_{Z} \wt V&=w V, \\
\wt\nabla_{\wt\nu} \wt V&=\nabla^{\mathcal{T}}_{\nu} V, \\
\wt\nabla_{\wt T} \wt V&=\nabla^{\mathcal{T}}_T V+\frac{2(n+1)}{n(n+2)}i(w-w')PV,
\end{align*}
where $\wt\nu$ and $\wt T$ are the horizontal lifts with respect to $(z^1, \dots, z^{n+1})$.
\end{prop}
\begin{lem}
The Tanaka--Webster connection associated with the CR scale $\sigma\in\calE(1, 0)$ satisfies
\[
\nabla\sigma=-\frac{2(n+1)}{n(n+2)}iP\theta\otimes\sigma.
\]
\end{lem}
\begin{proof}
We take an admissible coframe $\{\theta^\a\}$ such that $\theta\wedge\theta^1\wedge\cdots\wedge\theta^n=\zeta:=(dz^1\wedge\cdots\wedge dz^{n+1})|_{TM}$. 
By $d\zeta=0$, we have
\[
-\omega_\a{}^\a\wedge\theta\wedge\theta^1\wedge\cdots\wedge\theta^n=0.
\]
We also have $\omega_\a{}^\a+\omega_{\ol\a}{}^{\ol\a}=0$ from $\det(h_{\a\ol\b})=1$, so $\omega_\a{}^\a$ is of the form $if\theta$ with a real function $f$. Since 
$d\omega_\a{}^\a\equiv(1/n)R h_{\a\ol\b}\theta^\a\wedge\theta^{\ol\b}$ mod $\theta$, we obtain $f=-(1/n) R$, and hence
\[
\nabla\sigma=\frac{1}{n+2}\omega_\a{}^\a\otimes\sigma
=-\frac{2(n+1)}{n(n+2)}iP\theta\otimes\sigma.
\]
\end{proof}

{\it Proof of Proposition \ref{tractor-conn-lift}}\quad  We write $V=\sigma^w\bar{\sigma}^{w'}V_0$ and $\wt V=(z^0)^w(z^{\ol 0})^{w'}\wt V_0$ with $V_0\in \mathcal{T}^{1, 0}$, 
$\wt{V}_0\in\wt{\mathcal{T}}^{1, 0}$ as before. Since $\wt\nabla_Z V_0=0$, we have $\wt\nabla_Z \wt V=w\wt V$. By the lemma above, we have 
\begin{align*}
\nabla^{\mathcal{T}}_\nu V=\nabla^{\mathcal{T}}_\nu(\sigma^w\bar\sigma^{w'}V_0) 
=\sigma^w\bar\sigma^{w'}\nabla^{\mathcal{T}}_\nu V_0.
\end{align*}
The last term is identified with the vector field
\[
(z^0)^w(z^{\ol0})^{w'}\wt\nabla_\nu \wt{V}_0=\wt\nabla_\nu \wt V.
\]
Similarly, we have
\begin{align*}
\nabla^{\mathcal{T}}_T V&=\nabla^{\mathcal{T}}_T(\sigma^w\bar\sigma^{w'}V_0) \\
&=\sigma^w\bar\sigma^{w'}
\Bigl(\nabla^{\mathcal{T}}_T V_0-\frac{2(n+1)}{n(n+2)}i(w-w')PV_0 \Bigr),
\end{align*}
and this is identified with
\begin{align*}
&\quad (z^0)^w(z^{\ol0})^{w'}\Bigl(\wt\nabla_T \wt{V}_0-\frac{2(n+1)}{n(n+2)}i(w-w')P\wt{V}_0 \Bigr) \\
&=\wt\nabla_{\wt T} \wt V-\frac{2(n+1)}{n(n+2)}i(w-w')P\wt V.
\end{align*}
Thus we obtain the proposition. \qed
\bigskip

Let $\{Z_\a, \xi\}$ be a Graham--Lee frame of $T^{1, 0}X$ for the Fefferman defining function $\r:=|\sigma|^{-2}\br$. We define $\mathbb{Z}_0, \mathbb{Z}_\a, \mathbb{Z}_\infty \in \mathcal{T}^{1, 0}(1, 0)$ by 
\begin{equation}\label{GL-lifted}
\mathbb{Z}_0:=Z, \quad \mathbb{Z}_\a:=\wt{Z}_\a, \quad \mathbb{Z}_\infty:=-\wt{\xi}-\frac{\kappa}{2}Z,
\end{equation}
where $\wt{Z}_\a, \wt\xi$ are horizontal lifts and $\kappa$ is the transverse curvature.
The correspondence between $\mathcal{T}^{1, 0}(w, w')$ and the CR tractor bundle $\calE^A(w, w')$ is given as follows:
\begin{theorem}
Let $(U, (z^1, \dots, z^{n+1}))$ be a local coordinate system and $|\sigma|^2\in\calE(1, 1)$ the associated pseudo-Einstein CR scale. Then, the bundle isomorphism 
\[
\begin{array}{ccc}
\Psi\colon\ \calE^A(w, w')|_{U\cap M} & \longrightarrow &  \mathcal{T}^{1, 0}(w, w')|_{U\cap M} \\
\rotatebox{90}{$\in$} &   & \rotatebox{90}{$\in$} \\
\lambda Z^A+\nu^\b W_\b^A+\varphi Y^A & \longmapsto & \lambda \mathbb{Z}_0+\nu^\b
\mathbb{Z}_\b+|\sigma|^{-2}\varphi\mathbb{Z}_\infty
\end{array}
\]
  is independent of the choice of coordinates $(z^1, \dots, z^{n+1})$, and preserves the connections: $\Psi^*\nabla^{\mathcal{T}}=\nabla$. When $(w, w')=(0, 0)$, it also preserves the bundle metrics: $\Psi^*\wt g_{A\ol B}=h_{A\ol B}$.
\end{theorem}
\begin{proof}
First, we observe that along $\mathcal{N}_0$ the ambient metric is given by the matrix
\[
\bigl(\wt g(\mathbb{Z}_A, \mathbb{Z}_{\ol B})\bigr)=
\begin{pmatrix}
0 & 0 & 1 \\
0 & h_{\a\ol\b} & 0\\
1 & 0 & 0 
\end{pmatrix},
\]
which agrees with the tractor metric. Hence $\Psi$ preserves the metrics. 

To show that $\Psi$ is independent of choice of coordinates, we will compute the transformation formula of $\{\mathbb{Z}_A\}$ under coordinate change. 
 Let $(\hat z^1, \dots, \hat z^{n+1})$ be another coordinate system. We can take a fractional power 
\[
(d\hat z^1\wedge \cdots\wedge d\hat z^{n+1})^{\frac{-1}{n+2}}=e^{-f}(dz^1\wedge\cdots\wedge dz^{n+1})^{\frac{-1}{n+2}}
\]
with a holomorphic function $f=(1/2)\U+iv$ so that the new pseudo-Einstein CR scale is given by $\wh\sigma:=e^{-f|_M}\sigma\in\calE(1, 0)$. We note that $Z_{\ol\a} f=0$ gives
\begin{equation*}
v_\a=-\frac{i}{2}\U_\a
\end{equation*}
and hence
\[
f_\a=\U_\a.
\]
Also, from
\begin{align*}
0=\ol\xi f&=\Bigl(N-\frac{i}{2}T\Bigr)\Bigl(\frac{1}{2}\U+iv\Bigr) \\
&=\frac{1}{2}(N\U+v_T)+i\Bigl(Nv-\frac{1}{4}\U_T\Bigr),
\end{align*}
we have
\begin{equation*}
N\U=-v_T=\frac{i}{n}(v_\a{}^\a-v_{\ol\a}{}^{\ol\a})=-\frac{1}{2n}\Delta_b\U, \quad 
Nv=\frac{1}{4}\U_T
\end{equation*}
and 
\begin{equation*}
\xi f=(\xi+\ol\xi)f=2Nf=-\frac{1}{2n}\Delta_b\U+\frac{i}{2}\U_T.
\end{equation*}
The transformation formulas for $\xi$ and $\kappa$ are given by
\begin{equation*}
\begin{aligned}
\wh\xi&=e^{-\U}(\xi+\U^\a Z_\a),  \\
\wh\kappa&=e^{-\U}(\kappa-2N\U-\U_\a\U^\a) \\
&=e^{-\U}(\kappa+\frac{1}{n}\Delta_b\U-\U_\a\U^\a).
\end{aligned}
\end{equation*}
First, it is easy see that $\wh{\mathbb{Z}}_0=\mathbb{Z}_0$
since both are the $(1, 0)$-part of the infinitesimal generator of the $\mathbb{C}^*$-action.
We will consider $\wh{\mathbb{Z}}_\a, \wh{\mathbb{Z}}_\infty$. Since the fiber coordinates are related as
$\hat z^0=e^{-f}z^0$, we have
\[
\frac{d\hat z^0}{\hat z^0}=\frac{dz^0}{z^0}-df.
\]
Thus, if we write $V^\sigma, V^{\wh\sigma}$ for the horizontal lifts of a vector field $V$ on $M$ with respect to $(z^i)$ and $(\hat z^i)$, they satisfy
\[
V^{\wh\sigma}=V^{\sigma}+(Vf)Z.
\]
Hence we have
\[
\wh{\mathbb{Z}}_\a=Z^{\wh\sigma}_\a=Z^\sigma_\a+f_\a Z=\mathbb{Z}_\a+\U_\a\mathbb{Z}_0,
\]
and 
\begin{align*}
\wh{\mathbb{Z}}_\infty&=-\wh{\xi}^{\,\wh\sigma}-\frac{\wh\kappa}{2}Z \\
&=-e^{-\U}(\wh{\xi}^\sigma+(\wh\xi f)Z)-\frac{\wh\kappa}{2}Z \\
&=-e^{-\U}\Bigl[\xi^\sigma+\U^\a Z^\sigma_\a+\Bigl(-\frac{1}{2n}\Delta_b\U+\frac{i}{2}\U_T+\U^\a \U_\a\Bigr)Z \\
&\qquad \qquad +\frac{1}{2}\Bigl(\kappa+\frac{1}{n}\Delta_b\U-\U_\a\U^\a\Bigr)Z\Bigr] \\
&=e^{-\U}\Bigl(\mathbb{Z}_\infty-\U^\a \mathbb{Z}_\a-\frac{1}{2}(\U_\a\U^\a+i\U_T)\mathbb{Z}_0\Bigr).
\end{align*}
These transformation formulas agree with those for $Z^A, W^A_\a, |\sigma|^2Y^A$. Thus, $\Psi$ is independent of  choice of $(z^i)$.

We will express $\nabla^{\mathcal{T}}\mathbb{Z}_A$ in terms of the Tanaka--Webster connection  and check that it coincides with the tractor connection \eqref{tractor-conn} computed with a pseudo-Einstein contact form. Since $g$ satisfies the approximate Einstein equation, we have the equation \eqref{kappa-M} and \eqref{kappa-N} for $\kappa$. In particular, the derivative of $\kappa$ is given by 
\begin{equation}\label{kappa-derivative}
\begin{aligned}
N\kappa|_M&=-\frac{1}{2n}\Delta_b\kappa+\kappa^2+\frac{1}{n}|A_{\a\b}|^2 \\
&=-\frac{1}{n^2}\Delta_b P+\frac{4}{n^2}P^2+\frac{1}{n}|A_{\a\b}|^2,\\
\xi\kappa|_M&=-\frac{1}{n^2}\Delta_b P+\frac{i}{n}\nabla_T P+\frac{4}{n^2}P^2+\frac{1}{n}|A_{\a\b}|^2.
\end{aligned}
\end{equation}
We also recall that for a pseudo-Einstein contact form $\th$, we have 
\begin{equation}\label{p-E-formula}
\begin{aligned}
iA_{\a\g,}{}^\g=(1/n)\nabla_\a R, 
\quad T_\a=-\frac{1}{n}\nabla_\a P, \\
S=-\frac{1}{n^2}\Delta_b P-\frac{1}{n^2}P^2+\frac{1}{n}|A_{\a\b}|^2.
\end{aligned}
\end{equation}
By \eqref{nabla-Z}, we have
\[
\nabla^{\mathcal{T}}_{Z_\b}\mathbb{Z}_0=\wt\nabla_{\wt Z_\b}Z=\wt Z_\b=\mathbb{Z}_\b, \quad \nabla^{\mathcal{T}}_{Z_{\ol\b}}\mathbb{Z}_0=\wt\nabla_{\wt Z_{\ol\b}}Z=0.
\]
Since $T=i(\ol\xi-\xi)$ and $\mathbb{Z}_0$ has weight $(1, 0)$, we have 
\begin{align*}
\nabla^{\mathcal{T}}_T \mathbb{Z}_0=\wt\nabla_T Z-\frac{2(n+1)}{n(n+2)}i P Z
=-i\xi-\frac{2(n+1)}{n(n+2)}i P Z
=i\mathbb{Z}_\infty-\frac{i}{n+2}P \mathbb{Z}_0.
\end{align*}
We denote by $\Gamma_{ij}{}^k$ the Christoffel symbols of Tanaka--Webster connection. By using \eqref{omega-rho-0}, \eqref{kappa-derivative}, and \eqref{p-E-formula}, we compute as
\begin{align*}
\nabla^{\mathcal{T}}_{Z_\b}\mathbb{Z}_\a&=
\wt\nabla_{\wt Z_\b}\wt Z_\a=\Gamma_{\b\a}{}^\g\mathbb{Z}_\g -iA_{\a\b}\mathbb{Z}_0, \\
\nabla^{\mathcal{T}}_{Z_{\ol\b}}\mathbb{Z}_\a&=
\wt\nabla_{\wt Z_{\ol\b}}\wt Z_\a=\Gamma_{\ol\b\a}{}^\g\wt{Z}_\g+h_{\a\ol\b}\wt\xi \\
&=\Gamma_{\ol\b\a}{}^\g\mathbb{Z}_\g-(1/n)Ph_{\a\ol\b}\mathbb{Z}_0-h_{\a\ol\b}\mathbb{Z}_\infty \\
&=\Gamma_{\ol\b\a}{}^\g\mathbb{Z}_\g-P_{\a\ol\b}\mathbb{Z}_0-h_{\a\ol\b}\mathbb{Z}_\infty, \\
\nabla^{\mathcal{T}}_{T}\mathbb{Z}_\a
&=\wt\nabla_{\wt T}\wt Z_\a-\frac{2(n+1)}{n(n+2)}iPZ_\a \\
&=\Gamma_{T\a}{}^\g\mathbb{Z}_\g +\frac{2}{n(n+2)}iP\mathbb{Z}_{\a}+\frac{2i}{n}(\nabla_\a P)\mathbb{Z}_0 \\
&=\Gamma_{T\a}{}^\g\mathbb{Z}_\g +i\Bigl(P_\a{}^\b-\frac{P}{n+2}\d_\a{}^\b\Bigr)\mathbb{Z}_{\b}-2i T_\a\mathbb{Z}_0, 
\end{align*}
and 
\begin{align*}
\nabla^{\mathcal{T}}_{Z_\b}\mathbb{Z}_\infty
&=\wt\nabla_{Z_\b}\Bigl(-\xi-\frac{\kappa}{2}Z\Bigr) 
=\frac{1}{n}P\mathbb{Z}_\b+\frac{1}{n}(\nabla_\b P)\mathbb{Z}_0 
=P_\b{}^\g \mathbb{Z}_\g-T_\b \mathbb{Z}_0, \\
\nabla^{\mathcal{T}}_{Z_{\ol\b}}\mathbb{Z}_\infty&=\wt\nabla_{Z_{\ol\b}}\Bigl(-\xi-\frac{\kappa}{2}Z\Bigr) 
=-iA_{\ol\b}{}^\g\mathbb{Z}_\g-\frac{1}{n}(\nabla_{\ol\b}P)\mathbb{Z}_0 
=-iA_{\ol\b}{}^\g \mathbb{Z}_\g+T_{\ol\b}\mathbb{Z}_0, \\
\nabla^{\mathcal{T}}_T\mathbb{Z}_\infty&=\wt\nabla_T\Bigl(-\xi-\frac{\kappa}{2}Z\Bigr)-\frac{2(n+1)}{n(n+2)}i P\Bigl(-\xi-\frac{\kappa}{2}Z\Bigr) \\
&=\Bigl(\frac{i}{n^2}P^2+\frac{i}{n^2}\Delta_b P-\frac{i}{n}|A_{\a\b}|^2\Bigr)\mathbb{Z}_0
+\frac{2i}{n}(\nabla^\b P)\mathbb{Z}_\b -\frac{i}{n+2}P\mathbb{Z}_\infty \\
&=-iS\mathbb{Z}_0-2i T^\g\mathbb{Z}_\g-\frac{i}{n+2}P\mathbb{Z}_\infty.
\end{align*}
Thus, $\Psi^*\nabla^{\mathcal{T}}$ coincides with the tractor connection.
 \end{proof}

%%%%%%%%%%%%%%%%%%%%%%%%%%%%%%%%%%%%%%%%%%%%%%%%%%%%%%%%%%%%%%%%%%%%%
\subsection{The CR $D$-operator and the CR Weyl tractor}
By using the bundle isomorphism $\Psi$, we can identify a weighted tractor 
\[
t\in\calE_{A_1\cdots A_p \ol A_1\cdots\ol A_{p'}}^{B_1\cdots B_q 
\ol B_1\cdots \ol B_{q'}}(w, w')
\]
with an ambient tensor $\tilde t$ on $\wt X_0$ homogeneous of degree $(w+p-q, w+p'-q')$. 
Then, the ambient description of the CR $D$-operator is given by 
\begin{equation}\label{D-ambient}
D_A t=(n+w+w')\wt\nabla_A \tilde t+Z_A\wt\Delta \tilde t,
\end{equation}
where $\wt\Delta=-\wt\nabla_A\wt\nabla^A$ is the K\"ahler Laplacian and the right-hand side is independent of the choice of the extension $\tilde t$ (\cite{CG}).
In fact, $D_A$ is characterized by the equation
\[
h^{P\ol Q}D_{(A|P|}D_{B)\ol Q} t=-Z_{(A}D_{B)}t,
\]
where 
\begin{align*}
D_{AB}t&=(Z_B Y_A-Z_A Y_B)wt+Z_B W_A^\a\nabla_\a t-Z_A W_B^\b \nabla_\b t, \\
D_{A\ol{B}}t&=Z_{\ol B}Y_A wt-Z_A Y_{\ol B} w' t+Z_{\ol B}W_A^\a\nabla_\a t-Z_A W_{\ol B}^{\ol\b}\nabla_{\ol\b}t \\
&\quad -Z_AZ_{\ol B}\Bigl(i\nabla_T t+\frac{w'-w}{n+2}Pt\Bigr) 
\end{align*}
are CR invariant operators called the {\it double $D$-operators} (\cite{G}). By using Proposition \ref{tractor-conn-lift}, one can see that the ambient descriptions of these operators are 
\[
D_{AB}t=Z_B\wt\nabla_A \tilde t-Z_A\wt\nabla_B \tilde t, \quad D_{A\ol B}t=Z_{\ol B}\wt\nabla_A \tilde t-Z_A \wt\nabla_{\ol B}\tilde t,
\]
which are independent of the choice of the extension $\tilde t$. The equation \eqref{D-ambient} follows from these formulas.

Next we consider the curvature tensor $\wt R_{A\ol B C\ol E}$ of the ambient metric. We define the {\it CR Weyl tractor}  $S_{A\ol BC\ol E}\in\calE_{A\ol BC\ol E}(-1, -1)$ by
\begin{align*}
S_{A\ol BC\ol E}&=W_A^\a W_{\ol B}^{\ol\b}\,\Omega_{\a\ol\b C\ol E}-\frac{1}{n-1}W_A^\a Z_{\ol B}\nabla^{\ol\g}\Omega_{\a\ol\g C\ol E}-\frac{1}{n-1}Z_A W_{\ol B}^{\ol\b}\nabla^\mu
\Omega_{\mu\ol\b C\ol E} \\
&\quad +Z_AZ_{\ol B}
\Bigl(\frac{1}{(n-1)^2}\nabla^{\ol\g}\nabla^\mu\Omega_{\mu\ol\g C\ol E}+\frac{1}{n-1}
P^{\ol\g\mu}\Omega_{\mu\ol\g C\ol E}\Bigr).
\end{align*}
Note that the CR Weyl tractor defined in \cite{CG} is $(n-1)$ times ours. The CR Weyl tractor is a CR invariant tractor, which is characterized by the following conditions:
\[
S_{A\ol BC\ol E}=\overline{S_{B\ol A E\ol C}}, \quad Z^A S_{A\ol BC\ol E}=0, \quad 
W^A_\a W^{\ol B}_{\ol\b}S_{A\ol BC\ol E}=\Omega_{\a\ol\b C\ol E}, \quad 
D^{\ol B}S_{A\ol BC\ol E}=0.
\]
By equations \eqref{ricci-flat}, \eqref{Z-R}, \eqref{D-ambient} and Proposition \ref{tractor-conn-lift}, we can check that the ambient curvature tensor satisfies these conditions. Thus we have
\begin{prop}[{\cite[Lemma 8.3]{CG}}]\label{R-S}
When $n\ge2$,  $\wt R_{A\ol B C\ol E}|_{\mathcal{N}}=S_{A\ol B C\ol E}.$
\end{prop}

%%%%%%%%%%%%%%%%%%%%%%%%%%%%%%%%%%%%%%%%%%%%%%%%%%%%%%%%%%%%%%%%%%%%%
\subsection{Relation between $\ol{\mathcal{I}}'_\Phi$ and $\Phi(\Theta)$}
We are now ready to prove the following theorem, which relates two global CR invariants associated with an invariant polynomial of degree $n$:
\begin{thm}\label{fp-I-prime}
Let $\Phi$ be an invariant polynomial of degree $n$. Let $\rho$ be a Fefferman defining function associated with a pseudo-Einstein contact form $\th$. Then we have
\[
{\rm fp}\int_{\rho>\e} \omega\wedge\Phi(\Theta)=-n\int_M \mathcal{I}'_\Phi.
\]
\end{thm}
\begin{proof}
We may assume that $\Phi=T_{m_1}\cdots T_{m_k}$. By Proposition \ref{Theta-W} (i) and \eqref{int-lie}, we have
\[
{\rm fp}\int_{\rho>\e} \omega\wedge\Phi(\Theta)=\int_M \mathcal{L}_N\bigl(\vartheta\wedge\Phi(W)\bigr).
\]
Near the boundary, we have $\mathcal{L}_N\vartheta=-\kappa\vartheta$ and $d\Phi(W)=0$, and hence
\begin{align*}
\mathcal{L}_N\bigl(\vartheta\wedge\Phi(W)\bigr)
&=-\kappa\vartheta\wedge\Phi(W)+\vartheta\wedge d\bigl(N\lrcorner\,\Phi(W)\bigr) \\
&=-\kappa\vartheta\wedge\Phi(W)+d\vartheta\wedge \bigl(N\lrcorner\,\Phi(W)\bigr)+
({\rm exact\ form}).
\end{align*}
We take local coordinates $(z^1, \dots, z^{n+1})$ and the associated fiber coordinate $z^0$ of $\wt X_0$. We work in the $(1, 0)$-frame $\{\mathbb{Z}_A\}$ defined by \eqref{GL-lifted} with a Graham--Lee frame $\{Z_\a, \xi\}$. The dual coframe $\{\boldsymbol{\theta}^A\}$ is given by
\[
\boldsymbol{\theta}^0=\frac{dz^0}{z^0}-\frac{\kappa}{2}\pa\rho, \quad 
\boldsymbol{\theta}^\a=\theta^\a, \quad \boldsymbol{\theta}^\infty=-\pa\rho.
\]
By \eqref{W-Omega} and Proposition \ref{R-S}, we have
\begin{align*}
\theta\wedge\Phi(W)|_{TM}&=\theta\wedge\Phi(\wt\Omega)|_{TM} \\
&=i^n S^{(m_1)}_{[\a_1\ol\b_1\cdots} \cdots S^{(m_k)}_{\quad \cdots\a_n\ol\b_n]}\theta\wedge
\theta^{\a_1}\wedge\theta^{\ol\b_1}\wedge\cdots\wedge\theta^{\a_n}\wedge\theta^{\ol\b_n} \\
&=S^{\Phi}\theta\wedge(d\theta)^n.
\end{align*}
Thus, since $\kappa|_M=(2/n)P$, we have
\begin{equation}\label{first-term}
-\int_M \kappa\vartheta\wedge\Phi(W)=-\frac{2}{n}\int_M P S^\Phi \theta\wedge(d\theta)^n.
\end{equation}
Next, we will calculate $d\theta\wedge \bigl(N\lrcorner\,\Phi(W)\bigr)\big|_{TM}$. It suffices to compute the terms involving $\theta$ in $\bigl(N\lrcorner\,\Phi(W)\bigr)\big|_{TM}$. The horizontal lift of $N$ is given by
\[
\wt N={\rm Re}\,\wt\xi=-\frac{1}{2}(\mathbb{Z}_\infty+\mathbb{Z}_{\ol\infty})-\frac{\kappa}{4}(\mathbb{Z}_0+\mathbb{Z}_{\ol0}).
\]
Since $\mathbb{Z}_0\lrcorner\,\wt\Omega=\mathbb{Z}_{\ol0}\lrcorner\,\wt\Omega=0$, we have
\[
N\lrcorner\,\Phi(W)=\wt N\lrcorner\,\Phi(\wt\Omega)=-\frac{1}{2}\bigl(\mathbb{Z}_\infty\lrcorner\,\Phi(\wt\Omega)
+\mathbb{Z}_{\ol\infty}\lrcorner\,\Phi(\wt\Omega)\bigr).
\]
From $\boldsymbol{\theta}^\infty|_{TM}=i\theta$, we observe that the term  involving $\theta$ in $\bigl(N\lrcorner\,\Phi(W)\bigr)\big|_{TM}$ is 
\begin{align*}
&\quad i\theta\wedge\bigl(\mathbb{Z}_{\ol\infty}\lrcorner\,\mathbb{Z}_\infty\lrcorner\,\Phi(\wt\Omega)\bigr)\big|_{TM} \\
&=i^n n^2 \wt R^{(m_1)}_{[\infty\ol\infty A_2\ol B_2\cdots}\cdots \wt R^{(m_k)}_{\quad \cdots A_n\ol B_n]}\,\theta\wedge\boldsymbol{\theta}^{A_2}\wedge\boldsymbol{\theta}^{\ol B_2}\wedge\cdots\wedge\boldsymbol{\theta}^{A_n}\wedge\boldsymbol{\theta}^{\ol B_n} \big|_{TM}\\
&=i^n n^2 \wt R^{(m_1)}_{[\infty\ol\infty \a_2\ol \b_2\cdots}\cdots \wt R^{(m_k)}_{\quad \cdots \a_n\ol \b_n]}\,\theta\wedge\theta^{\a_2}\wedge\theta^{\ol \b_2}\wedge\cdots\wedge\theta^{\a_n}\wedge\theta^{\ol \b_n}.
\end{align*}
Thus, by using Proposition \ref{diff-form} (iii), we have
\begin{align*}
d\theta\wedge \bigl(N\lrcorner\,\Phi(W)\bigr)\big|_{TM}
&=i\theta\wedge d\theta\wedge\bigl(\mathbb{Z}_{\ol\infty}\lrcorner\,\mathbb{Z}_\infty\lrcorner\,\Phi(\wt\Omega)\bigr)\big|_{TM} \\
&=\frac{i}{n!(n-1)!}\Lambda^{n-1}\bigl(\mathbb{Z}_{\ol\infty}\lrcorner\,\mathbb{Z}_\infty\lrcorner\,\Phi(\wt\Omega)\bigr)\big|_{TM} \theta\wedge(d\theta)^n \\
&=\frac{i}{n!(n-1)!}\bigl((n-1)!\bigr)^2(-i)^{n-1}i^n n^2  \\
&\qquad\qquad \cdot h^{\a_2\ol\b_2}\cdots h^{\a_n\ol\b_n}
\wt R^{(m_1)}_{[\infty\ol\infty \a_2\ol \b_2\cdots}\cdots \wt R^{(m_k)}_{\quad \cdots \a_n\ol \b_n]}
\theta\wedge(d\theta)^n \\
&=-n S^\Phi_{\infty\ol\infty}\theta\wedge(d\theta)^n.
\end{align*}
From this and \eqref{first-term}, we obtain
\[
\int_M \mathcal{L}_N\bigl(\vartheta\wedge\Phi(W)\bigr)
=\int_M \Bigl(-\frac{2}{n}P S^\Phi-n S^\Phi_{\infty\ol\infty}\Bigr)\theta\wedge(d\theta)^n
=-n\int_M \mathcal{I}'_\Phi
\]
and complete the proof.
\end{proof}

\end{document}